\documentclass[11pt]{amsart}
\usepackage{xcolor}

\definecolor{grn}{rgb}{0,0.6,0}
\definecolor{mrn}{rgb}{0.3,0,0}
\definecolor{blue}{rgb}{0,0,0.7}
\definecolor{Mygray}{rgb}{0.75,0.75,0.75}
\definecolor{auburn}{rgb}{0.43, 0.21, 0.1}
\definecolor{britishracinggreen}{rgb}{0.0, 0.26, 0.15}
\definecolor{taupe}{rgb}{0.28, 0.24, 0.2}
\usepackage{amsmath,amssymb}
\newtheorem{theorem}{Theorem}[section]
\newtheorem{propn}[theorem]{Proposition}
\newtheorem{cor}[theorem]{Corollary}
\newtheorem{lemma}[theorem]{Lemma}

\newtheorem{rmk}[theorem]{Remark}

\newcommand{\Z}{\mathbb{Z}}
\newcommand{\Q}{\mathbb{Q}}

\newcommand{\Mod}[1]{\ (\mathrm{mod}\ #1)}
\newcommand{\p}{\mathfrak{p}}
\newcommand{\q}{\mathfrak{q}}
\newcommand{\mr}{\mathfrak{r}}
\newcommand{\ma}{\mathfrak{a}}
\newcommand{\mb}{\mathfrak{b}}
\newcommand{\mG}{\mathfrak{G}}
\newcommand{\mK}{\mathfrak{K}}
\newcommand{\tL}{\tilde{L}}
\newcommand{\tG}{\tilde{G}}

\usepackage{latexsym,amssymb}
\usepackage{amssymb}
\usepackage{graphicx}
\usepackage{ textcomp }
\usepackage[left=1.75cm,right=1.75cm,top=1.75cm,bottom=1.75cm]{geometry}
\usepackage{stmaryrd}
\usepackage{tikz}
\usepackage{tikzpagenodes} 
\usepackage{lipsum}
\usepackage{multicol}
\usepackage{enumitem}
\usepackage{booktabs}

\begin{document}
\baselineskip=14.5pt
\title[$\Z_2$-extension of some biquadratic number fields]{Stability of 2-class groups in the $\Z_2$-extension of certain real biquadratic fields }

\author{H Laxmi and Anupam Saikia}
\address[H Laxmi and Anupam Saikia]{Department of Mathematics, Indian Institute of Technology Guwahati, Guwahati - 781039, Assam, India}

\email[H Laxmi]{hlaxmi@iitg.ac.in}

\email[Anupam Saikia]{a.saikia@iitg.ac.in}
\renewcommand{\thefootnote}{}



\footnote{\emph{We confirm that all the data are included in the article.}}

\renewcommand{\thefootnote}{\arabic{footnote}}
\setcounter{footnote}{0}

\begin{abstract}
Greenberg's conjecture on the stability of $\ell$-class groups in the cyclotomic $\mathbb{Z}_{\ell}$-extension of a real field has been proven for various infinite families of real quadratic fields for the prime $\ell=2$. In this work, we consider an infinite family of real biquadratic fields $K$. With some extensive use of elementary group theoretic and class field theoretic arguments, we investigate the $2$-class groups of the $n$-th layers $K_n$ of the cyclotomic $\Z_2$-extension of $K$ and verify Greenberg's conjecture. We also relate capitulation of ideal classes of certain sub-extensions of $K_n$ to the relative sizes of the $2$-class groups.  
\end{abstract}

\maketitle
\section{Introduction}

\noindent The class group of a number field measures how far its ring of integers from being a unique factorization domain. In order to examine class groups, it is fairly common to break them as the product of their $\ell$-Sylow subgroups, also known as $\ell$-class groups, where $\ell$ is a prime, and study them individually. If $F/K$ is a Galois extension of number fields, then the class group $\mathcal{C}l_F$ of $F$ is equipped with a Galois module structure. Iwasawa examined the growth of $\ell$-class groups of number fields in an infinite tower of Galois extensions extension by viewing them as Galois modules. Given a prime number $\ell$, an infinite Galois extension $K_{\infty}$ of $K$ is said to be a $\Z_{\ell}$-extension if $\Gamma := {\rm{Gal}}(K_{\infty}/K)$ is topologically isomorphic to $\Z_{\ell}$, the group of $\ell$-adic integers. For any $n \geq 0$, there exists a unique intermediate field $K_n$, also known as the {\it $n$-th layer}, with $\Gamma_n: = {\rm{Gal}}(K_n/K)$ isomorphic to $\Z/{\ell}^n\Z$ (cf. \cite[Proposition 13.1]{washington_book}). With respect to the projection maps from $\Gamma_n$ to $\Gamma_m$ for $n \geq m$, the Iwasawa algebra is defined as $\Lambda:= \displaystyle\lim_{\substack{\longleftarrow \\ n}}\Z_{\ell}[\Gamma_n]$, and is isomorphic to the power series ring over $\Z_{\ell}$. The norm map between the class groups of $K_n$ and $K_m$ for $n \geq m$ can be restricted to the $\ell$-class groups, denoted by $A(K_n)$ and $A(K_m)$, respectively. The ring $\Z_{\ell}[\Gamma_n]$ acts on $A(K_n)$ by the extension of the action of $\Gamma_n$ on $A(K_n)$. Thus, $\Lambda$ acts canonically on the inverse limit $X(K_{\infty}):=  \displaystyle\lim_{\substack{\longleftarrow \\ n}}A(K_n)$ which is known as the Iwasawa module. This inverse system is formed with respect to the norm $N_{n,m}: A(K_n)\rightarrow A(K_m)$ for $n\geq m$. Though it is clear that information about $A(K_n)$ provides information on $X(K_{\infty})$, Iwasawa theory suggests that one can study $A(K_n)$'s via $X(K_{\infty})$. This was substantiated by Iwasawa's class number formula \cite{iwasawa}, which states that there exist constants $\mu(K_{\infty}/K), \lambda(K_{\infty}/K)$, and $\nu(K_{\infty}/K)$ such that for sufficiently large $n$, if $\ell^{e(n)}$ is the largest power of $\ell$ that divides $\#A(K_n)$, then $$e(n) = \mu(K_{\infty}/K)\ell^n + \lambda(K_{\infty}/K)n + \nu(K_{\infty}/K).$$  
The constants $\mu(K_{\infty}/K), \lambda(K_{\infty}/K)$, and $\nu(K_{\infty}/K)$ are known as the Iwasawa invariants. One of the frequently explored $\Z_{\ell}$-extensions is the cyclotomic $\Z_{\ell}$-extension. Let $\zeta_{m}$ be an $m$-th root of unity. For an odd prime $\ell$, the $n$-th layer $\Q_n$ for $\Q$ in such an extension is the unique subfield of $\Q(\zeta_{l^{n+1}})$ of degree $\ell^n$ over $\Q$. For $\ell = 2$, $\Q_n$ is the maximal real subfield of $\Q(\zeta_{2^{n+2}})$. In case of cyclotomic $\Z_{\ell}$-extension of an arbitrary number field $K$, $K_n$ is the compositum of $K$ and $\Q_n$. Due to Brumer's proof of Leopoldt's conjecture for abelian extensions of $\Q$ in \cite{brumer}, any real number field abelian over $\Q$ has a unique $\Z_{\ell}$-extension, namely, the cyclotomic $\Z_{\ell}$-extension.

In \cite{greenberg}, Greenberg conjectured that if $K$ is a totally real number field, then the invariants $\mu(K_{\infty}/K)$ and $\lambda(K_{\infty}/K)$ corresponding to the cyclotomic $\Z_{\ell}$-extension of $K$ must be equal to zero. Ferrero and Washington (\cite{ferrero-washington}) proved that $\mu(K_{\infty}/K)$ vanishes for the cyclotomic $\Z_{\ell}$-extension of a number field $K$ abelian over $\Q$. Greenberg's conjecture is still open, with partial progress made by considering particular values of $\ell$ and specific families of number fields. Some of the works in this direction can be seen in \cite{fukuda}, \cite{fukuda-komatsu}, \cite{gras2}, \cite{kumakawa2}, \cite{L-S3}, \cite{L-S_arx}, \cite{mizusawa2}, \cite{mizusawa3}, \cite{mizusawa4}, \cite{mouhib}, \cite{mouhib-mova}, \cite{mouhib-mova2}, \cite{nishino}, \cite{ozaki-taya}.

By virtue of results from class field theory, the $\ell$-class group of $K$ is isomorphic to the Galois group (over $K$) of the maximal unramified abelian $\ell$-extension of $K$, also known as the $\ell$-Hilbert class field of $K$. If $L(K)$ is the $\ell$-Hilbert class field and $\tL(K)$ is the maximal unramified $\ell$-extension of $K$, then ${\rm{Gal}}(L(K)/K)$ is the maximal abelian quotient of ${\rm{Gal}}(\tL(K)/K)$. Observing the same for the $n$-th layers $K_n$ and passing on to the inverse limit, it can be observed that $X(K_{\infty})$ is the maximal abelian quotient of ${\rm{Gal}}(\tL(K_{\infty})/K_{\infty}) = \displaystyle\lim_{\substack{\longleftarrow \\ n}}{\rm{Gal}}(\tL(K_n)/K_n)$. This highlights the connection between Iwasawa modules and class field towers. In recent years problems based on Greenberg's conjecture have been solved using techniques involving class field towers. Interested readers may refer to \cite{chems24}, \cite{mizusawa2}, \cite{mizusawa4}, \cite{mizusawa5}.   

When $\ell = 2$, for the cyclotomic $\Z_2$-extension of a number field $K$, the base layer is given by $K_0 = K$, $K_1 = K(\sqrt{2})$, and if $K_n = K(\sqrt{a_n})$, then $K_{n+1} = K(\sqrt{2 + a_n})$. Though Greenberg's conjecture is not completely settled for real quadratic fields, lately even multiquadratic number fields with discriminant having small number of prime factors are being investigated. In this work, we focus on the cyclotomic $\Z_2$-extension of an infinite family of real biquadratic fields given by $K = \Q(\sqrt{p}, \sqrt{r})$ where the constituent primes satisfy Condition 1, namely 
\begin{equation} \label{Cond}
    p \equiv 9 \Mod{16}, \ r \equiv 3 \Mod{4},\ \left( \dfrac{p}{r} \right) = -1, \ \left(\dfrac{2}{p}\right)_{4} = -1.
\end{equation}
Our aim is to validate Greenberg's conjecture for this family of fields by looking into rank and order stability. By stability, we mean the eventual termination of growth of the order and rank of $\ell$-class groups. It is already known that the Iwasawa invariants $\mu$ and $\lambda$ are equal to 0 for the $\Z_2$-extensions of the subfields $\Q(\sqrt{p})$, $\Q(\sqrt{r})$, and $\Q(\sqrt{pr})$ (cf. \cite{mouhib-mova}, \cite{ozaki-taya}, \cite{fukuda-komatsu}, \cite{nishino}). In this article, we record how the 2-class groups of the subfields $\Q_n(\sqrt{r}), \Q_n(\sqrt{p})$, and $\Q_n(\sqrt{pr})$ of $K_n$ influence the structure of its 2-class groups. To achieve this, we require genus theory and results on class field towers that terminate. We shall resort to some theorems on capitulation in ramified and unramified extensions, and also thoroughly study the action of Galois groups of certain extensions on class groups. We shall also utilize the family of fields $K$ under our consideration to know more about the ray class field modulo primes involving $2, p$, and $r$ of the $n$-th layers $\Q_n(\sqrt{r})$ and $\Q_n(\sqrt{p})$ of the fields $\Q(\sqrt{r})$ and $\Q(\sqrt{p})$, respectively. We now state our results:

\begin{theorem}\label{Thm1}
Let $K = \Q(\sqrt{p}, \sqrt{r})$ with $p \equiv 9 \Mod{16}, \ r \equiv 3 \Mod 4,\ \left( \dfrac{p}{r} \right) = -1$, and $\left(\dfrac{2}{p}\right)_{4} = -1$. Suppose 
$K_{\infty}$ is the $\Z_2$-extension of $K$ with $n$-th layers $K_n$. Then, the corresponding Iwasawa module $X(K_{\infty}) = \displaystyle\lim_{\substack{\longleftarrow \\ n}}A(K_n)$ is finite and cyclic, with $\#A(\Q_n(\sqrt{p})) \leq \#A(K_n) \leq 2\cdot\#A(\Q_n(\sqrt{p}))$ for all $n \geq 0$. In particular, the corresponding Iwasawa invariant $\lambda(K_{\infty}/K)$ vanishes for $K$.
\end{theorem} 
\noindent As the Iwasawa invariants $\mu$ and $\lambda$ of $\Q(\sqrt{p})$ corresponding to its $\Z_2$-extension are equal to 0, there exists $n_0 \geq 0$ such that $\#A(\Q_n( \sqrt{p})) = \#A(\Q_{n_0}( \sqrt{p}))$ for all $n \geq n_0$. With this $n_0$, from Theorem \ref{Thm1}, we conclude:

\begin{cor}\label{Cor1}
Let $n_0 \geq 0$ with $\#A(\Q_n( \sqrt{p})) = \#A(\Q_{n_0}( \sqrt{p}))$ for all $n \geq n_0$. Then, $\#X(\Q_{\infty}(\sqrt{p}))\leq \#X(K_{\infty}) \leq 2^{n_0 + 1}\cdot\#X(\Q_{\infty}(\sqrt{p}))$, where $X(\Q_{\infty}(\sqrt{p}))$ is the Iwasawa module corresponding to the $\Z_2$-extension of $\Q(\sqrt{p})$. 
\end{cor} 

\noindent Let $S$ be a finite set of prime integers coprime to 2. For any number field $\mathfrak{f}$, suppose $A_S(\mathfrak{f})$ denotes the 2-part of the ray class group of $\mathfrak{f}$ modulo the product of all the prime ideals of $\mathfrak{f}$ that divide $p$ in $S$. Let $(\Q_{\infty})_S$ be the maximal pro-2-extension of $\Q_{\infty}$ unramified outside $S$. In an attempt to classify the sets $S$ for which ${\rm{Gal}}((\Q_{\infty})_S/\Q_{\infty})$ is prometacyclic, Mizusawa (cf. \cite[Theorem 3.1]{mizusawa4}) obtained a result that helps to evaluate the 2-rank of $A_S(\mathfrak{f})$ via certain quadratic extensions $F/\mathfrak{f}$. As an outcome of that result and Theorem \ref{Thm1}, we obtain the following corollary.

\begin{cor}\label{Cor2}
Suppose $p \equiv 9 \Mod{16}, \ r \equiv 3 \Mod{4}, \ \left( \dfrac{p}{r} \right) = -1$, and $\left(\dfrac{2}{p}\right)_{4} = -1$. Then we have ${\rm{rank}}_2A_{\{p\}}(\Q_n(\sqrt{r})) = 2$.   
\end{cor}

\noindent Although $\#A(K_n) \geq \#A(\Q_n(\sqrt{p}))$ as $K_n/\Q_n(\sqrt{p})$ is a ramified extension, it would be interesting to know when these orders are equal. Our next theorem is aimed in this direction.

\begin{theorem}\label{Thm2}
Let $K = \Q(\sqrt{p}, \sqrt{r})$ with $p \equiv 9 \Mod{16}, \ r \equiv 3 \Mod 4,\ \left( \dfrac{p}{r} \right) = -1$, and $\left(\dfrac{2}{p}\right)_{4} = -1$. For any $n \geq 0$, $L(K_n)$, the 2-Hilbert class field of $K_n$ is always abelian over $\Q_n(\sqrt{p})$. For $n =0$, $\#A(K_0) = \#A(\Q_{0}(\sqrt{p})) = 1$. Furthermore, if $n \geq 1$, then $\#A(K_n) = \#A(\Q_{n}(\sqrt{p}))$ if and only if some non-trivial ideal class of $A(\Q_n(\sqrt{p}))$ capitulates in $A(K_n)$. 
\end{theorem}

\noindent As the field $L(K_n)$ is ramified at primes above 2 and $r$ over $\Q_n(\sqrt{p})$, we obtain the next corollary.

\begin{cor}\label{Cor3}
Let $p \equiv 9 \Mod{16}, \ r \equiv 3 \Mod 4,\ \left( \dfrac{p}{r} \right) = -1$, and $\left(\dfrac{2}{p}\right)_{4} = -1$. Let $\Q_n(\sqrt{p})$ be the $n$-th layer in the $\Z_2$-extension of $\Q(\sqrt{p})$. Then, the 2-rank of the maximal abelian extension of $\Q_n(\sqrt{p})$ unramified outside the primes above 2 and $r$ is at least 2.
\end{cor}

\noindent Finally, we deal with the case $n=1$ and provide an alternate condition for the equality $\#A(K_1) = \#A(\Q_1(\sqrt{p}))$.

\begin{theorem}\label{Thm3}
 Let $K = \Q(\sqrt{p}, \sqrt{r})$ with $p \equiv 9 \Mod{16}$, $r \equiv 3 \Mod 4$, $\left( \dfrac{p}{r} \right) = -1$, and $\left(\dfrac{2}{p}\right)_{4} = -1$. Suppose $a$ and $b$ are nonzero integers such that $p = a^2 - 2b^2$. If $\left( \dfrac{a}{p} \right) = -1$, or equivalently, $\left( \dfrac{b}{p} \right) = 1$, then $\#A(K_1) = \#A(\Q_1(\sqrt{p})) = 2$. 
\end{theorem}

\begin{rmk}
In light of Theorem \ref{Thm2} and Theorem \ref{Thm3}, if $p$ satisfies the said conditions, then a non-trivial ideal of $A(\Q_1(\sqrt{p}))$ capitulates in $A(K_1)$.    
\end{rmk}

\section{Preliminaries}

\noindent 
Throughout the article, for any number field $F$, we shall use $A_{\ell}(F)$ to denote the $\ell$-class group of $F$, and $A(F)$ for the 2-class group in particular. By $(2^{a_1}, \ldots, 2^{a_n})$, we mean the finite abelian group of the form $\Z/2^{a_1}\Z \oplus \cdots \oplus \Z/2^{a_n}\Z$, and the symbol `$\cong$' will be used to denote that two groups are isomorphic. 
We denote the action of $\sigma \in {\rm{Gal}}(F/K)$ on $[x] \in \mathcal{C}l_F$ by $[x]^{\sigma}$, which is also equal to $ [\sigma(x)]$. When $F/K$ is a cyclic extension of degree $\ell^n$ for any prime $\ell$, Chevalley derived a useful formula involving the subgroup of $A_{\ell}(F)$ fixed by the Galois action and the group $A_{\ell}(K)$. This formula is widely known as the genus formula, and we state it here for quadratic extensions. 

\begin{theorem}(\cite{chevalley}, \cite[Theorem 2.5]{mizu_thesis})\label{genusfor}
Let $F/K$ be a quadratic extension of number fields with Galois group $G = {\rm{Gal}}\left(F/K\right)$. Let $A(F)^{G}$ be the subgroup of $A(F)$ consisting of the ideal classes that are fixed by the action of $G$ on $A(F)$. Let $N_{F/K}$ denote the norm map from $F$ to $K$. Let $E(F)$ and $E(K)$ be the unit groups of $F$ and $K$, respectively. If $t$ is the number of places of $K$ ramified in $F$, then 
\begin{align*}
\#A(F)^{G} &= \#A(K) \times \dfrac{2^{t-1}}{\left[ E(K):E(K)\cap N_{F/K}(F^{\times}) \right]}.
\end{align*}
Let $B(F)^G$ be the subgroup of $A(F)^G$ consisting of ideal classes that contain an ideal fixed by the action of $G$ (also known as the strongly ambiguous ideal classes). Then,
\begin{align*}
\#B(F)^{G} &= \#A(K) \times \dfrac{2^{t-1}}{\left[ E(K): N_{F/K}(E(F)) \right]}.
\end{align*}
\end{theorem}

\noindent By $\text{rank}_{\ell} A_{\ell}(K)$ or $\ell$-rank of $A_{\ell}(K)$, we mean the dimension of $A_{\ell}(K)/\ell A_{\ell}(K)$ as a vector space over the field $\Z/\ell\Z$. The following proposition connects the 2-rank of $A(K)$ with the genus formula.

\begin{propn}\cite[Proposition 2.2]{L-S3}\label{rmk to genus}
Suppose $F/K$ is a quadratic extension of number fields such that the image of the lifting map $j: A(K) \rightarrow A(F)$ is trivial. Then, the nontrivial element of $G={\rm{Gal}}(F/K)$ acts as $-1$ on $A(F)$. In that case, $A(F)^{G}$ is the subgroup of elements of order $2$. Consequently, $$\# A(F)^{G} = \# \left( A(F)/2A(F)\right) = 2^{\ {\rm{rank}}_2A(F)}.$$
\end{propn}

\noindent If the class number of $K$ is odd, then its $2$-class group is trivial. For a quadratic extension $F/K$ with such a $K$, Proposition \ref{rmk to genus} implies that  
\begin{equation*}
2^{\ {\rm{rank}}_2A(F)} = \dfrac{ 2^{t-1}}{\left[E(K):E(K)\cap N_{F/K}(F^{\times})\right]},
\end{equation*}

For any number field $F$, we let $L(F)$ to be the 2-Hilbert class field of $F$. One of the important subfields of $L(F)$ containing $F$ is the genus field $F_G$ of $F$. For an arbitrary extension of number fields $F/K$, the genus field $F_G$ of $F$ is defined to be the maximal field of type $LF$ such that $L/K$ is abelian and $LF/F$ is unramified. For an abelian extension $F/\Q$, the genus field of $F$ is the maximal unramified extension of $F$ that is abelian over $\Q$, and hence is naturally contained in the Hilbert class field $H(F)$ of $F$. If $F/\Q$ is a quadratic extension, then the narrow genus field $F_G^{+}$ is the maximal abelian extension of $\Q$ contained in the narrow $2$-Hilbert class field of $F$. If $F$ is a real quadratic field, $F_G$ is the maximal real subfield of $F_G^{+}$ and otherwise, $F_G = F_G^{+}$ (cf. \cite[Chapter 6, Theorem 3.2]{janusz}). When $F = \Q(\sqrt{d})$, we express the prime factorization of the discriminant of $F$ by $D_{F} = \pm 2^{e}p_1^{\ast}\cdots p_t^{\ast}$, where $e = 0,2$ or $3$, and
\begin{equation*}
  p_i^{\ast} = 
      \begin{cases}
        p_i & \text{if} \  p_i \equiv 1 \Mod 4\\
        -p_i & \text{if}\ p_i \equiv 3 \Mod 4.
      \end{cases}     
\end{equation*}
In this notation, from \cite[Chapter 6, Theorem 3.10]{janusz}, we have $F_G^{+} = \mathbb{Q}( \sqrt{d}, \sqrt{p_1^{\ast}}, \cdots, \sqrt{p_t^{\ast}})$. Moreover, for quadratic extensions, from Theorem 3.3(b) of \cite[Chapter 6]{janusz}, $F_G$ is the largest 2-elementary, unramified extension of $F$ and thus the $2$-rank of Gal($F_G/F$) is equal to the $2$-rank of $A(F)$. From this definition, it is trivial that $K = \Q(\sqrt{p}, \sqrt{r})$ is the genus field of $k = \Q(\sqrt{pr})$, where $p$ and $r$ satisfy Condition (\ref{Cond}). Inductively, we can also show that for any $n \geq 0$, $K_n/k_n$ is an unramified extension, where $K_n$ and $k_n$ are the $n$-th layers of the $\Z_2$-extension of $K$ and $k$, respectively. Fukuda and Komatsu (cf. \cite{fukuda-komatsu}), and then Nishino (cf. \cite{nishino}) studied the $\Z_2$-extension of $k = \Q(\sqrt{pr})$ with $p$ and $r$ satisfying Condition (\ref{Cond}) and proved that such fields follow Greenberg's conjecture. They explicitly found the group structure of $X(k_{\infty})$ using methods that involved ray class groups and surjectivity of the norm maps under certain conditions. We now merge their results into one and state it as follows.

\begin{theorem}(\cite[Theorem 2.2, Proposition 3.4]{fukuda-komatsu}, \cite[Theorem 2]{nishino})\label{f-k}
Let $k = \Q(\sqrt{pr})$, $p \equiv 9 \Mod{16}, \ r \equiv 3 \Mod 4, \left( \dfrac{p}{r} \right) = -1$, and $\left(\dfrac{2}{p}\right)_{4} = -1$. The Iwasawa $\lambda$-invariant corresponding to the $\Z_2$-extension of $k$ is equal to 0. For $n \geq 1$, if $A(k_n)$ denotes the 2-class group of $k_n$, then $A(k_n) \cong (2,2)$. Furthermore, the 2-Hilbert class field of $k_n$ for $n \geq 1$ is given by $L(k_n) = \Q_n(\sqrt{p}, \sqrt{r}, \sqrt{p_1})$, where $p = p_1p_2$ with $p_i \in \Q_1 = \Q(\sqrt{2})$ and $p_i > 0$.   
\end{theorem}

Considering the primes $p$ and $r$ individually, we present the results on the Iwasawa modules of $\Q(\sqrt{r})$ and $\Q(\sqrt{p})$, respectively.

\begin{theorem}\cite[Section 2]{ozaki-taya}\label{A(Qn_r)}
For $r \equiv 3 \Mod{4}$, the Iwasawa module $X(\Q(\sqrt{r})_{\infty})$ corresponding to the $\Z_2$-extension of $\Q(\sqrt{r})$ is trivial.   
\end{theorem}

\begin{theorem}\cite[Theorems 3.8, 4.1]{mouhib-mova}\label{mouhib-mova}
Let $p \equiv 1 \Mod{8}$ and $\left( \dfrac{2}{p} \right)_4 = -1$. Then, the Iwasawa module $X(\Q(\sqrt{p})_{\infty})$ is cyclic with $\lambda = 0$.    
\end{theorem}

One of the most intriguing aspects of a $\Z_{\ell}$-extension is that the prime(s) above $\ell$ is(are) the only ramified prime(s) in the extension. Also, there exists $n_0 \geq 0$ such that each ramified prime is totally ramified in $F_{\infty}/F_{n_0}$ (cf. \cite[Proposition 13.2, Lemma 13.3]{washington_book}). The vanishing of Iwasawa $\lambda$ and $\mu$ invariants is closely associated with the order and rank stability of the $\ell$-class groups. We now state a result proved by Fukuda that involves stability and has been a key ingredient in many works involving the study of Iwasawa modules. 

\begin{theorem}\cite[Theorem 1]{fukuda}\label{fukuda's result}
Let $\ell$ be a prime number. Let $F$ be a number field and let $F_{\infty}/F$ be a $\Z_{\ell}$-extension of $F$. Let $A_{\ell}(F_n)$ denote the $\ell$-class group of $F_n$ in the extension $F_\infty/F$. Let $n_0 \geq 0$ be an integer such that any prime of $F_{\infty}$ that is ramified in $F_{\infty}/F$ is totally ramified in $F_{\infty}/F_{n_0}$. Then the following statements hold. 
\begin{enumerate}
\item If there exists an integer $n \geq n_0$ such that $\#A_{\ell}(F_{n+1}) = \#A_{\ell}(F_n)$, then $\#A_{\ell}(F_m) = \#A_{\ell}(F_n)$ for all $m \geq n$. In particular, both Iwasawa invariants $\mu(F_\infty/F)$ and $\lambda(F_\infty/F)$ vanish. 
  
  \smallskip
  
\item If there exists an integer $n \geq n_0$ such that ${\rm{rank}}_{\ell}A_{\ell}(F_{n+1}) = {\rm{rank}}_{\ell}A_{\ell}(F_{n})$, then ${\rm{rank}}_{\ell}A_{\ell}(F_{m}) = {\rm{rank}}_{\ell}A_{\ell}(F_{n})$ for all $m \geq n$. In particular, the Iwasawa invariant $\mu(F_\infty/F)$ vanishes.
\end{enumerate}
\end{theorem}

\noindent\textbf{A consequence of Burnside's basis theorem}: Let $G$ be an $\ell$-group for a prime $\ell$. Let $G^{\prime}$ denote the commutator subgroup of $G$, and $\phi(G)$ be the intersection of all maximal subgroups of $G$. Then, $G^{\prime} \subseteq \phi(G)$. Burnside's basis theorem states that any lift of generators of $G/\phi(G)$ will generate $G$. If $G$ is a pro-$\ell$-group such that $G/G^{\prime}$ is procyclic, then so will $G/\phi(G)$ be. As a consequence of Burnside's basis theorem, this implies that $G$ is procyclic. Applying this result on the $\ell$-class field tower of $F$, we deduce that if $L(F)/F$ is a cyclic extension, then $\tL(F)/F$ is also cyclic, and hence an abelian extension. This implies that $\tL(F) = L(F)$, and the $\ell$-class field tower of $F$ terminates at $L(F)$.

Next, we state the result by Gorenstein (cf. \cite[Chapter 5, Theorem 4.5]{gorenstein}) which characterizes all the groups $G$ of order $2^m$, where $m\geq 3$ and $G/G^{\prime} \cong (2,2)$. This result is essential because it helps us determine the 2-groups associated with 2-class field towers.

\begin{theorem}\cite[Theorem 1]{kisilevsky} \label{gorenstein result}
Let $G$ be a finite group of order $2^m$ with $m \geq 3$. Suppose $G^{\prime}$ denotes the commutator subgroup of $G$, and $G/G^{\prime} \cong (2,2)$. Then, $G$ is one of $D_{2^m}$ (dihedral group of order $2^m$), $Q_{2^m}$ (generalized quaternion group of order $2^m$), or $S_{2^m}$ (semi-dihedral group of order $2^m, m>3$). 
\end{theorem}  

While studying cyclic unramified extension of number fields of degree $\ell$, where $\ell$ is an odd prime, Taussky (cf. \cite{taussky cond}) introduced two conditions in terms of capitulation of ideal classes. These were further utilized by Kisilevsky for $\ell = 2$, and are now widely known as ``Taussky conditions''. Let $F/K$ be a cyclic unramified extension of a prime degree $\ell$, $j : \mathcal{C}l_K \rightarrow \mathcal{C}l_F$ be the lifting map and $N_{F/K}: \mathcal{C}l_F \rightarrow \mathcal{C}l_K$ be the norm map, respectively. Then, the Taussky conditions are given by:
\begin{enumerate}[label = (\Alph*)]
\item $\#\left( Ker(j) \cap N_{F/K}(\mathcal{C}l_F) \right) >1$
\item $\#\left( Ker(j) \cap N_{F/K}(\mathcal{C}l_F) \right) =1$.
\end{enumerate}

\noindent If $F$ is a number field whose 2-class group is isomorphic to $(2,2)$, then Taussky and Furtwrangler (cf. \cite{furtwangler}, \cite{taussky cft}) proved that the 2-class field tower of $F$ terminates at the first level $F^{(1)} = L(F)$ or at the second level $F^{(2)} = L( F^{(1)})$ of the 2-class field tower. If it is the latter case, then there exists an intermediate field $L \subseteq \tL(F)$ such that $L/F$ is an unramified non-abelian extension of degree 8. Using Taussky conditions along with certain cohomological and group theoretic arguments, Kisilevsky proved a remarkable result which helps us identify ${\rm{Gal}}(L/F)$ and ${\rm{Gal}}(F^{(2)}/F)$.

\begin{theorem}\label{Kisilevsky 1}\cite[Theorem 2]{kisilevsky}
Let $\mK$ be a number field with $A(\mK) \cong (2,2)$. Suppose $\mK^{(1)} = L(\mK)$, $\mK^{(2)} = L(\mK^{(1)})$, and $\mK_i$ and $L$ be fields satisfying:
$$\mK ~ \subset ~ \mK_1, ~ \mK_2, ~ \mK_3 ~ \subset ~ \mK^{(1)} ~ \subset ~ L ~ \subseteq ~ \mK^{(2)}.$$
Let $j_i: \mathcal{C}l(\mK) \rightarrow \mathcal{C}l(\mK_i)$ be the lifting map for $i = 1,2,3$. Then,
\begin{enumerate}
\item If $\mK^{(2)} = \mK^{(1)}$, then for each $i \in \{1,2,3\}$, $\#Ker(j_i) = 4$ and each field $\mK_i$ satisfies Condition (A).
\item If ${\rm{Gal}}(L/\mK) \cong Q_8$, where $Q_8$ is the quaternion group of order 8, then for each $i$, $\mK_i$ satisfies Condition (A) with $\#Ker(j_i) = 2$. Also, $L = \mK^{(2)}$.
\item  If ${\rm{Gal}}(L/\mK) \cong D_8$, where $D_8$ is the dihedral group of order 8, then $\mK_1, \mK_2$ satisfy Condition (B) and $\#Ker(j_1) = \#Ker(j_2)= 2$. In addition, the following hold:
\begin{enumerate}
\item If $\mK_3$ satisfies Condition (B), then $\#Ker(j_3)= 2$, and ${\rm{Gal}}(\mK^{(2)}/\mK) \cong S_{2^m}$. 
\item If $\mK_3$ satisfies Condition (A) and $\#Ker(j_3)= 2$, then ${\rm{Gal}}(\mK^{(2)}/\mK) \cong Q_{2^m}$. 
\item If $\mK_3$ satisfies Condition (A) and $\#Ker(j_3)= 4$, then ${\rm{Gal}}(\mK^{(2)}/\mK) \cong D_{2^m}$.
\end{enumerate}
\end{enumerate}
\end{theorem}
\noindent It should be noted from the proof that all these cases are mutually exclusive and exhaustive. Hence, these are the only possible combinations of Taussky conditions and order of kernels.

Mouhib and Movahhedi proved a result on real number fields whose maximal unamified 2-extension is either the quaternion group or the semidihedral group of order $2^m, m \geq 3$. We now state it as follows. 

\begin{theorem}\cite[Theorem 3.1]{mouhib-mova2}\label{mouhib-mova 2008}
Let $F$ be a number field with $X(F_{\infty}) \cong (2,2)$. Let $n_0$ be the smallest integer such that $F_{\infty}/F_{n_0}$ is totally ramified at the prime(s) above 2, and $A(F_{n_0}) \cong (2,2)$. If ${\rm{Gal}}(\tL(F_{n_0})/F_{n_0})$ is either the quaternion group or the semidihedral group, then ${\rm{Gal}}(\tL(F_{\infty})/F_{\infty}) \cong {\rm{Gal}}(\tL(F_{n_0})/F_{n_0})$. 
\end{theorem}

\noindent Later, Mizusawa explicitly studied the possibility of ${\rm{Gal}}(\tL(F_{\infty})/F_{\infty})$ and the semidihedral group being isomorphic, and proved the following result.

\begin{theorem}\cite[Theorem 1]{mizusawa3}\label{never sd}
    Let $F$ be a real quadratic field. Then, the Galois group of the maximal unramified pro-2-extension of $F_{\infty}$ which is ${\rm{Gal}}(\tL(F_{\infty})/F_{\infty})$ is not the semidihedral group.  
\end{theorem}
In addition to the groups mentioned in Theorem \ref{Kisilevsky 1}, we also require some properties of groups of order 16. It is known that there are 14 distinct groups of order 16 up to isomorphism (cf. \cite{conrad16}). Of these, the abelian groups are of the form $(16), (8,2), (4,4), (4,2,2)$, and $(2,2,2,2)$. Using generators, relations, and the order of the group elements (cf. \cite{conrad16}, \cite{clausen16}), we draw some conclusions on some of the nonabelian groups of order 16 which will be essential at a later stage. 

\begin{rmk}\label{groups 16}
    \begin{enumerate}
        \item The groups $D_{16}, Q_{16}$, and $S_{16}$ do not have a subgroup of type $(2,4)$.
        \item The group $(4) \rtimes (4)$ (the semidirect product of $(4)$ with itself) has only one subgroup of type $(2,2)$.
        \item The group $M_{16}$ (modular group of order 16) has only 3 subgroups of order 8. These are $(8)$ (occurring twice) and $(2,4)$. Hence, this group can have only 3 quotients of type $(2)$.
        \item The group $(2,2)\rtimes(4)$ has no normal cyclic subgroup of order 4. 
    \end{enumerate}
\end{rmk}

An ideal $\mb$ contained in $\mathcal{O}_K$ is said to capitulate in an extension $F/K$ if $\mb\mathcal{O}_F$ is a principal ideal in $\mathcal{O}_F$. This phenomenon is called capitulation, and the ideal class $[\mb]$ then belongs to the kernel of the lifting map from $K$ to $F$. In general, obtaining the ideal classes that capitulate in an extension is challenging, one of the reasons being the presence of the mysterious units in the ring of integers. In case of a cyclic unramified extension, Hilbert's Theorem 94 (cf. \cite[Theorem 1.8.1]{lemmermeyer}) asserts capitulation of a nontrivial ideal class, though it does not provide the exact number of classes that capitulate. The next result by Gras is vital for understanding capitulation in ramified extensions.

\begin{theorem}\cite[Theorem 1.1, part 2]{gras3}\label{GrasCap}
 Let $F/K$ be a cyclic, totally ramified $\ell$-extension of degree ${\ell}^N$, $N \geq 1$. Let $G = {\rm{Gal}}(F/K) = \langle \psi \rangle$, $\ell^{(e(F))}$ be the exponent of $A(F)$, and let $m(F)$ be the minimal integer such that $(\psi - 1)^{m(F)}$ annihilates $A(F)$. Let $[y] \in A(F)$ be of order ${\ell}^e$, annihilated by $(\psi - 1)^{m}$, where, for $s \in [0, N-1]$, if $m \in [\ell^s, \ell^{s+1} -1]$, then $e \in [1, N-s]$. In that case, $[x] = N_{F/K}([y]) \in A(K)$ capitulates in $A(F)$.
\end{theorem}

Finally, we state Hecke's quadratic reciprocity law, which is similar to the usual quadratic reciprocity law, but can be considered in number fields with some additional assumptions. This will enable us to study the decomposition of certain prime ideals in fields of higher degree over $\Q$.

\begin{theorem}\cite[Corollary 10.13]{roy-cft}\label{Reciprocity}
 Let $n=2$, $F$ be a number field, and $\left( \frac{x}{y} \right)$ denote the Legendre symbol for $x, y \in F^{\times}$. Let $\sigma_1, \ldots, \sigma_R$ be all the real embeddings of $F$. Fix $x,y \in F^{\times}$ with $x_i = \sigma_i(x)$ and $y_i = \sigma_i(y)$ for $i=1, \ldots, R$. If $x$ and $y$ are relatively prime and not divisible by prime(s) above $2\Z$, and either of them is congruent to a square modulo 4, then, $$\left( \dfrac{x}{y} \right) = \left(\displaystyle{\prod_{i=1}^{R}(-1)^{s(x_i,y_i)} }\right)\cdot \left( \dfrac{y}{x} \right), \text{ where,} \ s(x_i, y_i) = \begin{cases}
     1, \ \text{if } x_i <0 \text{ and } y_i < 0;\\
     0, \ \text{otherwise.}
 \end{cases}.$$
\end{theorem}

\section{The rank of $A(K_n)$}

\noindent This section onwards, we fix $K = \Q(\sqrt{p}, \sqrt{r})$ and $k = \Q(\sqrt{pr})$ where $p$ and $r$ follow Condition (\ref{Cond}). For $p \equiv 1 \Mod 4$, the genus field of $\Q(\sqrt{p})$ is itself. Since ${\rm{rank}}_2A(\Q(\sqrt{p})) = {\rm{rank}}_2{\rm{Gal}}(\Q(\sqrt{p})_G/\Q(\sqrt{p})) = 0$,  $\Q(\sqrt{p})$ does not have a nontrivial unramified 2-extension, and $\#A(\Q_0(\sqrt{p})) = \#A(\Q(\sqrt{p})) = 1$. When $p \equiv 1 \Mod 8$ with $\left( \dfrac{2}{p} \right)_4 = -1$, then it has been shown in \cite[Theorem 4.1]{mouhib-mova} that $\#A(\Q_1(\sqrt{p})) = 2$. We now delve into $A\left(\Q_2(\sqrt{p})\right)$ with some more conditions on $p$.

\begin{propn}\label{A(Q_n(p)} 
Let $p \equiv 9 \Mod {16}$ with $\left( \dfrac{2}{p} \right)_4 = -1$. Then, there exists $p_1 \in \Q_1 = \Q(\sqrt{2})$ such that $L(\Q_1(\sqrt{p})) = \Q_1(\sqrt{p}, \sqrt{p_1})$, and $\#A\left(\Q_2(\sqrt{p})\right) \leq 4$.
\end{propn}
\begin{proof} 
If $p \equiv 9 \Mod{16}$, then from \cite{fukuda-komatsu, nishino}, there exist totally positive prime elements $p_1, p_2$ in $\Q_1$ such that $p = p_1p_2$. Furthermore, if $\left( \dfrac{2}{p} \right)_4 = -1$, then from \cite[Lemma 1]{nishino} and the proof of \cite[Theorem 2]{nishino}, the prime above 2 is unramified in $\Q_1(\sqrt{p_i})/\Q_1$. Thus, $\Q_1(\sqrt{p}, \sqrt{p_1})/\Q_1(\sqrt{p})$ is an unramified abelian extension of degree 2, and $L(\Q_1(\sqrt{p})) = \Q_1(\sqrt{p}, \sqrt{p_1})$.

\noindent From \cite[Proposition 3.6]{mouhib-mova}, we note that when $p \equiv 9 \Mod {16}$ and $\left( \dfrac{2}{p} \right)_4 = -1$, the Iwasawa module corresponding to $\Q(\sqrt{p})$ is cyclic if and only if we have $A(M) \cong (2,2)$, where $M = \Q\left(\sqrt{p(2 + \sqrt{2})}\right)$. Since the cyclicity of the module is already established, due to the conditions taken on $p$, we obtain $A(M) \cong (2,2)$. Employing \cite[Proposition 3.5]{L-S_arx} for the field $\Q_2(\sqrt{p})$ and using $A(M) \cong (2,2)$, we derive $\#A(\Q_2(\sqrt{p})) \leq 1/2 \cdot \#A(\Q_1(\sqrt{p}))\cdot \#A(M) = 1/2 \cdot 2 \cdot 4 = 4$. 
\end{proof}

Hereafter, for $p \equiv  9 \Mod{16}$ and $\left( \dfrac{2}{p} \right)_4 = -1$, $p_1$ and $p_2$ will always denote the factors of $p$ in $\Q_1$ as mentioned in the proof of Proposition \ref{A(Q_n(p)}. We now prove a lemma regarding the order of $A(\mK)$ for $\mK = \Q(\sqrt{pr})$, where the primes $p$ and $r$ satisfy $p \equiv 1 \Mod{4}$ and $r \equiv 3 \Mod{4}$. A particular case of this result can be found in \cite{nishino}, although we present a more general version here for the sake of completeness.

\begin{lemma}\label{A_(k0) = 2}
Let $\mK = \Q(\sqrt{pr})$ such that $p \equiv 1 \Mod{4}$, $r \equiv 3 \Mod{4}$, and $\left( \dfrac{p}{r} \right) = -1$. Then, $\#A(\mK) = 2$.
\end{lemma}
\begin{proof}
As $\mK$ is real, its genus field $\mK_G$ is the field $\Q( \sqrt{p}, \sqrt{r})$, which is a quadratic extension of $\mK$. Since the 2-ranks of ${\rm{Gal}}(\mK_G/\mK)$ and $A(\mK)$ are equal, we deduce that ${\rm{rank}}_2A(\mK) = 1$, and thus $A(\mK)$ is cyclic.

\noindent Let $\mr$ be the prime ideal above $r\Z$ in $\mK$. Then, since $r\Z$ is ramified in $\mK$, the order of the ideal class $[ \mr ]$ is at most 2 in $A(\mK)$. The ideal $r\Z$ is inert in $\Q(\sqrt{p})$ as $\left( \dfrac{p}{r} \right) = -1$, and thus $\mr$ is inert in $\mK_G/ \mK$. On that account, it is clear that $\mr$ does not split in an unramified extension of $\mK$, and hence it is not principal in $\mK$. Therefore, $o[\mr] = 2$. 

\noindent From \cite[Theorem 3.3(b)]{janusz}, Gal($\mK_G/\mK$) is isomorphic to $A(\mK)/A(\mK)^2$. Since $\mr$ does not split in $\mK_G/ \mK$, the corresponding Forbenius element $\left(\dfrac{\mK_G/ \mK}{\mr}\right)$ must be non-trivial in Gal($\mK_G/ \mK$). Hence, $[\mr] \not\in A(K)^2$. Now, $A(\mK) = A(\mK)^2 \cup [\mr]A(\mK)^2 = \{ id, [\mr] \}A(\mK)^2$ implies that $A(\mK) =  \{ id, [\mr] \}$ by Nakayama's lemma.
\end{proof}

For $k = \Q(\sqrt{pr})$ where $p$ and $r$ satisfy Condition 1, it is obvious from Lemma \ref{A_(k0) = 2} that $\#A(k) = 2$, and that its genus field is $K = \Q(\sqrt{p}, \sqrt{r})$ which is the field of our interest. As a first step, we study $A(K)$ and the rank of $A(K_1)$, where $K_1 = K(\sqrt{2})$.

\begin{lemma}\label{rank A(K1)}
Let $K = \Q(\sqrt{p}, \sqrt{r})$ where $p$ and $r$ satisfy Condition 1. Then $A(K)$ is the trivial group, and ${\rm{rank}}_2A(K_1) = 1$, where $K_1 = K(\sqrt{2})$.
\end{lemma}
\begin{proof}
From Lemma \ref{A_(k0) = 2}, since $A(k) \cong (2)$, $k_G = K$ must be the 2-Hilbert class field of $k$. As $K/k$ is a cyclic extension, according to Burnside's basis theorem, $\tL(k) = K = \tL(K)$, and $A(K) = \{ id \}$.

\noindent Viewing $K$ as a biquadratic extension over $\Q$ with subfields $\Q(\sqrt{p}), \Q(\sqrt{r})$, and $k$, it can be verified that the prime ideal $2\Z$ of $\Q$ has two prime factors in $K$. In turn, these factors are the only ramified primes in the extension $K_1/K$. Since $A(K)$ is trivial, the genus formula for $K_1/K$ provides the upper bound: ${\rm{rank}}_2A(K_1) \leq 1$.

\noindent As stated in Theorem \ref{f-k}, the 2-Hilbert class field of $k_1 = \Q(\sqrt{2}, \sqrt{pr})$ is $ L(k_1) = \Q(\sqrt{2}, \sqrt{p}, \sqrt{r}, \sqrt{p_1} )$. We note that the field $K_1$ is an unramified extension of $k_1$, and is properly contained in $L(k_1)$. Therefore, $L(k_1)/K_1$ is an unramified quadratic extension and thus, ${\rm{rank}}_2A(K_1) \geq 1$. Integrating this with the previous argument, we obtain ${\rm{rank}}_2A(K_1) = 1$.
\end{proof}

In \cite[Lemma 2.1]{kumakawa2}, it can be found that for a particular infinite family of quadratic fields denoted by $F$, $\#A(F_{n+1}) \leq \#A(F_n)\cdot \#A(F_n^{\prime})$. There, $F_n$ and $F_{n+1}$ are the $n$-th and $(n+1)$-th layers of the $\Z_2$-extension of $F$ respectively, and $\Q_n \subset F_n^{\prime} \neq F_n \subset F_{n+1}$. In our work, we shall adapt those arguments for the biquadratic extension $K_n = \Q_n(\sqrt{p}, \sqrt{r})/\Q_n$. We note that this extension does not involve layers from the cyclotomic $\Z_2$-extension of $K$ other than $K_n$. 
\begin{center}
\begin{figure}[hbt!] 
 \begin{tikzpicture}

    \node (Q1) at (0,0) {$\Q_n$};
    \node (Q2) at (3,2) {$k_n = \Q_n(\sqrt{pr})$};
    \node (Q3) at (0,2) {$\Q_n(\sqrt{p})$};
    \node (Q4) at (-3,2) {$\Q_n(\sqrt{r})$};  
    \node (Q5) at (0,4) {$K_{n} = \Q_n(\sqrt{p}, \sqrt{r})$};
    
    \draw (Q1)--(Q2);
    \draw (Q1)--(Q3); 
    \draw (Q1)--(Q4);
    \draw (Q2)--(Q5);
    \draw (Q3)--(Q5);
    \draw (Q4)--(Q5);
    
    \node (R1) at (2.1,3.1) {$\langle \sigma\tau \rangle$};
    \node (R2) at (0.4, 3.1) {$\langle \tau \rangle$};
    \node (R3) at (-2, 3.1) {$\langle \sigma \rangle$};
     \end{tikzpicture}
    \end{figure}
\end{center}

\begin{lemma}\label{bound of A(K_n)}
Let $K = \Q(\sqrt{p}, \sqrt{r})$ where $p$ and $r$ satisfy Condition (\ref{Cond}). Let $K_n$ be the $n^{th}$-layer of $K$ in the $\Z_2$-extension of $K$. Suppose ${\rm{Gal}}(K_n/\Q_n(\sqrt{p})) = \langle \tau \rangle$ and ${\rm{Gal}}(K_n/\Q_n(\sqrt{r})) = \langle \sigma \rangle$. Then, the following hold for $n \geq 1$.
\begin{enumerate}
\item $\#A(K_n) = \#A(K_n)^{\tau + 1} \cdot \#A(K_n)^{\langle \sigma\tau \rangle} = 2\cdot\#A(K_n)^{\tau + 1}$.
\item $\#A(\Q_n(\sqrt{p})) \leq \#A(K_n) \leq 2\cdot\#A(\Q_n(\sqrt{p}))$.
\end{enumerate}
\end{lemma}
\begin{proof}
From Theorem \ref{A(Qn_r)}, $\#A(\Q_n(\sqrt{r})) = 1$ for every $n \geq 0$, and we deduce from Proposition \ref{rmk to genus} that $\sigma$ acts as inverse on $A(K_n)$. This leads to the equality of the subgroups $A(K_{n})^{\sigma\tau -1}$ and $A(K_{n})^{\tau +1}$ of $A(K_n)$. Here, $A(K_{n})^{\sigma\tau -1} = \{ [\mathfrak{a}]^{\sigma\tau}\cdot [\mathfrak{a}]^{-1}: [\mathfrak{a}] \in A(K_{n}) \}$ and $A(K_{n})^{\tau +1}$ is defined similarly.

\noindent (1). From the exact sequence
$$1 \longrightarrow A(K_{n})^{ \langle \sigma\tau \rangle} \longrightarrow A(K_{n}) \longrightarrow A(K_{n})^{\sigma\tau -1} \longrightarrow 1,$$
we obtain $\#A(K_{n}) = \#A(K_{n})^{ \langle \sigma\tau \rangle} \cdot \#A(K_{n})^{\sigma\tau -1} = \#A(K_{n})^{ \langle \sigma\tau \rangle}\cdot\#A(K_{n})^{\tau +1}$. We apply genus formula for $K_n/k_n$ as ${\rm{Gal}}(K_n/k_n) = \langle \sigma\tau \rangle$. As this extension is unramified, and $\#A(k_n) = 4$ for all $n \geq 1$ (from Theorem \ref{f-k}), $ \#A(K_n)^{\langle \sigma\tau \rangle} \leq 2^{-1}\cdot 4 = 2$ irrespective of the action of $\tau$. We now look at two possibilities:

\noindent{\bf Case 1}. {\bf $\tau$ acts as inverse on $A(K_n)$}: If $[ \mathfrak{P}]^{\tau} = [\mathfrak{P}]^{-1}$ for every $[ \mathfrak{P}] \in A(K_n)$, then $A(K_n)^{\tau +1} = \{ id \}$, and hence, $\#A(K_n) =  \#A(K_n)^{\langle \sigma\tau \rangle}$ and thus $\#A(K_n) \leq 2$. But from Lemma \ref{rank A(K1)}, ${\rm{rank}}_2A(K_1) = 1$ implies ${\rm{rank}}_2A(K_n) \geq 1$, and $\#A(K_n) \geq 2$. Therefore, in this case, $\#A(K_n) = 2 = 2\cdot \#A(K_{n})^{\tau +1}$.

\noindent{\bf Case 2}. {\bf $\tau$ does not act as inverse on $A(K_n)$}: In this case, $A(K_n)^{\tau +1}$ will be a non-trivial subgroup of $A(K_n)$. Since $A(K_n)$ is an abelian 2-group, so is $A(K_n)^{\tau +1}$. Therefore, by Cauchy's theorem on finite groups, there exists $[\q] \in A(K_n)^{\tau +1}$ of order 2. It is clear that $[\q]^{\tau} = [\q]$, and $[\q]^{\sigma\tau} = ([\q]^{-1})^{\tau} = [\q]^{\tau} = [\q]$. This implies $[\q] \in A(K_n)^{\langle \sigma\tau \rangle}$, and that $\#A(K_n)^{\langle \sigma\tau \rangle} \geq 2$. From Case 1, it is apparent that $\#A(K_n)^{\langle \sigma\tau \rangle} \leq 2$, and hence,  $\#A(K_n)^{\langle \sigma\tau \rangle} = 2$. This proves that $\#A(K_n) = 2\cdot\#A(K_n)^{\tau + 1}$.

\vspace{0.2cm}

\noindent (2). Let $N$ be the norm map from $A(K_n)$ to $A(\Q_n(\sqrt{p}))$. The map $N$ is surjective as the extension $K_n/\Q_n(\sqrt{p})$ is ramified at primes above 2 and $r$ leading to the first inequality $\#A(\Q_n(\sqrt{p})) \leq \#A(K_n)$. We now observe the kernel of $N$ by taking an element $[\mathfrak{P}] \in Ker(N)$. By Chebotarev's density theorem, we may choose $\mathfrak{P}$ to be a prime ideal in $K_n$ lying above a split prime $\p$ in $\Q_n(\sqrt{p})$. Now, $\p \subseteq \mathfrak{P}\mathfrak{P}^{\tau} \cap \mathcal{O}_{\Q_n(\sqrt{p})} = \langle \alpha \rangle$, where $\alpha \in \mathcal{O}_{\Q_n(\sqrt{p})}\setminus \{ 0\}$. If $\alpha$ is a unit, then $1 \in \mathfrak{P}\mathfrak{P}^{\tau}$, which is not possible. Hence, $\alpha$ is not a unit, and $\p = \langle \alpha \rangle$ as $\p$ is maximal in $\mathcal{O}_{\Q_n(\sqrt{p})}$. Therefore, $\p\mathcal{O}_{K_n} = \mathfrak{P}\mathfrak{P}^{\tau}$ is principal, and thus, $[\mathfrak{P}]^{1 + \tau} = id$. Hence, $Ker(N) \subseteq T:=\{ [\mathfrak{P}] \in A(K_n) : [\mathfrak{P}]^{1 + \tau} = id \}$, and $\#A(K_n)^{\tau + 1} = \#(A(K_n)/T) \leq \#(A(K_n)/Ker(N)) = \#A(\Q_n(\sqrt{p}))$. From part 1, $\#A(K_n) = 2\cdot \#A(K_n)^{\tau + 1} \leq 2\cdot\#A(\Q_n(\sqrt{p}))$.
 \end{proof}

\begin{rmk}\label{A(tau +1) is cyclic}
From Lemma \ref{bound of A(K_n)}, we observe that $A(K_n)^{\langle \sigma\tau \rangle} \cong (2)$. Therefore, there exists $[\ma]$ of order 2 in $A(K_n)$        such that $A(K_n)^{\langle \sigma\tau \rangle} = \{ id, [\ma] \}$. If $[\q] \in A(K_n)$ is of order two such that $[\q] \in A(K_n)^{\tau + 1}$, then from Case 2 of Lemma \ref{bound of A(K_n)}, $[\q] \in A(K_n)^{\langle \sigma\tau \rangle}$. Therefore, $[\q] = [\ma]$. This implies that $A(K_n)^{\tau + 1}$ can have at most one element of order 2, and if it does, then it must be $[\ma]$. Hence, $A(K_n)^{\tau + 1}$ is cyclic of order $\#A(K_n)/2$.
\end{rmk}

\noindent Suppose $[\mathfrak{P}] \neq [\ma] \in A(K_n)$ is another element of order 2, and let $H = \{ id, [\mathfrak{P}] \}$ be another subgroup of $A(K_n)$. Then, $H$ and $A(K_n)^{\tau +1 }$ intersect trivially. Consequently, $\#(H\cdot A(K_n)^{\tau +1}) = \#A(K_n)$, and hence $H\cdot A(K_n)^{\tau +1} = A(K_n)$. Since by Remark \ref{A(tau +1) is cyclic} $A(K_n)^{\tau +1}$ is cyclic, the group $A(K_n) = H\cdot A(K_n)^{\tau +1}$ can have at most three elements of order 2. This implies that ${\rm{rank}}_2A(K_n) \leq 2$ for all $n \geq 0$. We shall now see that the rank is always 1 for $n \geq 1$.

\subsection*{Proof of Theorem \ref{Thm1}}

\begin{proof}
From Theorem \ref{f-k}, since $A(k_n) \cong (2,2)$ for every $n \geq 1$, $K_n \subset L(k_n) = \Q_n(\sqrt{p}, \sqrt{r}, \sqrt{p_1}) \subseteq \tL(k_n)$ for every $n$. As $K_n/k_n$ is unramified, we obtain $\tL(K_n) = \tL(k_n)$. Thus, we have the tower of field extensions:
$$\Q ~ \subset ~ \Q_1 ~ \subset ~ \Q_2 ~ \subset ~ k_2 ~ \subset ~ \mathfrak{K}_1, ~ \mathfrak{K}_2, ~ \mathfrak{K}_3 ~ \subset L(k_2) ~ \subset L(K_2) ~ \subseteq  ~ \tL(k_2) = \tL(K_2),$$
where, $\mathfrak{K}_1 = \Q(\sqrt{pr}, \sqrt{p_1})$, $\mathfrak{K}_2 = \Q(\sqrt{pr}, \sqrt{p_2})$, and $\mathfrak{K}_3 = K_2$. From Theorem \ref{gorenstein result}, $\tG_2 = {\rm{Gal}}(\tL(k_2)/k_2)$ has three subgroups of index two, and those are $\mG_i = {\rm{Gal}}(\tL(k_2)/\mathfrak{K}_i), \ i = 1,2,3$. By group theory, it can be seen that the possibilities of the subgroups $\mG_i$ of $\tG_2$ are $(2,2), \ \Z/2^{r}\Z ~ (r \in \{1,2, m-1 \} ), \ D_{2^{m-1}}$, and $Q_{2^{m-1}}$ (cf. \cite[Pages 27, 28]{mizu_thesis}). As $\mathfrak{K}_i/k_2$ is unramified for each $i$, $\tL(\mathfrak{K}_i) = \tL(k_2)$. Hence, the largest abelian quotient of $\mG_i$ will correspond to the maximal, abelian unramified extension of $\mathfrak{K}_i$, that is, $A(\mathfrak{K}_i) = \mG_i/\mG_i^{\prime}$. Again using Theorem \ref{gorenstein result}, $\mG_i/\mG_i^{\prime}$ is either $(2,2)$ or $\Z/2^{r}\Z ~ (r \in \{1,2, m-1 \})$. Out of these, if $\tG_2 \cong Q_8$ or $(2,2)$, then all the $A(\mathfrak{K}_i)$'s are cyclic, and isomorphic to each other. In particular, $A(K_2) = A(\mathfrak{K}_3)$ is cyclic. 

If $\tG_2$ is not isomorphic to $Q_8$ or $(2,2)$, there exist $i_1 \neq i_2 \neq i_3  \in \{1,2,3\}$ such that $A(\mathfrak{K}_{i_1}) \cong A(\mathfrak{K}_{i_2}) \cong (2,2)$, and $A(\mathfrak{K}_{i_3})$ is a cyclic 2-group. Consider $p_1$, $p_2$ as defined in Lemma \ref{rank A(K1)}. There exist  integers $a$ and $b$ such that $p_1 = a + b\sqrt{2}$, and $p_2 = a - b\sqrt{2}$. The minimal polynomial of $\sqrt{p_1}$ over $\Q$ is given by $f(x):= (x^2 - a)^2 - 2b^2$. Suppose $\alpha$ is a root of $f$. Substituting $\alpha^2$ with $t$ in $f(\alpha) = 0$, we note that $t$ can take the values $p_1$ as well as $p_2$. Therefore, $\alpha$ can assume the values $\sqrt{p_1},  -\sqrt{p_1},\sqrt{p_2}$, and $-\sqrt{p_2}$. Since $\pm\sqrt{p_2} \not\in \mathfrak{K_1}$, $\mathfrak{K_1}$ is not a Galois extension of $\Q$, and similarly, $\mathfrak{K_2}$ is also not a Galois extension of $\Q$. But, $\mathfrak{K_1}$ and $\mathfrak{K_2}$ are isomorphic as field extensions over $\Q$ as $\pm \sqrt{p_1}$ and $\pm \sqrt{p_2}$ are all conjugates over $\Q$. As a result, the class groups of $\mathfrak{K_1}$ and $\mathfrak{K_2}$, and in particular, $A(\mathfrak{K}_1)$ and  $A(\mathfrak{K}_{2})$ must be isomorphic. This implies that $A(\mathfrak{K}_3) = A(K_2)$ is always cyclic. From Lemma \ref{rank A(K1)}, $A(K_1)$ is cyclic, and from Theorem \ref{fukuda's result}, $A(K_n)$ is cyclic for all $n\geq 1$. Thus, $X(K_{\infty})$ is a cyclic module.

\noindent From Part 2 of Lemma \ref{bound of A(K_n)}, we recall that $\#A(K_n) \leq 2\cdot \#A(\Q_n(\sqrt{p}))$. Also, from Theorem \ref{mouhib-mova}, $\lambda$-invariant of $\Q(\sqrt{p})$ is 0 implies $\#A(\Q_n(\sqrt{p}))$ is bounded as $n$ grows. Combining both, we conclude that $\#A(K_n)$ is bounded as $n$ tends to infinity. Thus, the Iwasawa $\lambda$-invariant of $K$ is equal to $0$.   
\end{proof}

\subsection*{Proof of Corollary \ref{Cor1}}
\begin{proof}
Suppose $n_0 \geq 0$ is the stage such that $\#A(\Q_n(\sqrt{p})) = \#A(\Q_{n_0}( \sqrt{p}))$ for all $n \geq n_0$. Then, it is straightforward from Lemma \ref{bound of A(K_n)} and Theorem \ref{fukuda's result}, that $\#A(K_n) = \#A(K_{n_{0}+1})$ for all $n \geq n_0+1$. The inequality $\#A(K_n) \leq 2 \cdot \#A(\Q_n(\sqrt{p}))$ yields $\#X(K_{\infty}) \leq 2^{n_0 + 1}\cdot\#X(\Q_{\infty}(\sqrt{p}))$.    
\end{proof}

\subsection*{Proof of Corollary \ref{Cor2}}
\begin{proof}
Let $S = \{ p \}$ and $\Sigma = \phi$ be the empty set. Then, $A_{\Sigma}(\Q_n(\sqrt{r})) = A(\Q_n(\sqrt{r})) = \{id\}$ from Theorem \ref{A(Qn_r)}, and ${\rm{rank}}_2A_{\Sigma}(K_n) = 1$ from Theorem \ref{Thm1}. The quadratic extension $K_n/\Q_n(\sqrt{r})$ is ramified at all primes in $S \setminus\Sigma = S$, and unramified outside $S$. Theorem 3.1 of \cite{mizusawa4} states that if $F/\mathfrak{f}$ is a quadratic extension unramified outside $S$ and ramified at all primes in $S\setminus \Sigma$ with $A_{\Sigma}(\mathfrak{f}) = \{ id \}$, then ${\rm{rank}}_2A_S(\mathfrak{f}) = 1 + {\rm{rank}}_2A_{\Sigma}(F)$. Therefore, it is immediate that ${\rm{rank}}_2A_{\{p\}}(\Q_n(\sqrt{r})) = 1 + {\rm{rank}}_2A_{\Sigma}(K_n) = 1 + {\rm{rank}}_2A(K_n) = 2$. 
\end{proof}

\begin{rmk}
As $A(K_n)$ is cyclic for each $n$, by Burnside's basis theorem, $\tL(K_n) = L(K_n)$. Thus, $\tL(k_n) =  \tL(K_n) = L(K_n)$. This also implies that $L(K_{\infty}) = \tL(K_{\infty})$ is the maximal unramified 2-extension of $k_{\infty}$, which clearly is a finite extension of $k_{\infty}$. 
\end{rmk}  

\section{An equivalent criteria for $\#A(K_n) = \#A(\Q_n(\sqrt{p}))$}
\noindent It is imperative to understand the action of $\tau$ on $A(K_n)$, where ${\rm{Gal}}(K_n/\Q_n(\sqrt{p})) = \langle \tau \rangle$ in view of part 1 of Lemma \ref{bound of A(K_n)}. Firstly, we note an elementary consequence of Lemma \ref{bound of A(K_n)}.

\begin{propn}\label{tau as -1}
For ${\rm{Gal}}(K_n/\Q_n(\sqrt{p})) = \langle \tau \rangle$, $\tau$ acts on $A(K_n)$ as inverse if and only if $\#A(K_n) = 2$.
\end{propn}
\begin{proof}
As given in Case 1 of Lemma \ref{bound of A(K_n)}, $\tau$ acting as inverse on $A(K_n)$ implies that $A(K_n)^{\tau + 1} = \{ id \}$ and $\#A(K_n) = 2$. Conversely, suppose $A(K_n) = \{ id, ~ [\mb] \}$. If $[\mb]^{\tau} = id$, then $\mb^{\tau}$ must be principal, which also implies that $\mb$ is principal, which is a contradiction. Therefore, $[\mb]^{\tau} = [\mb] = [\mb]^{-1}$. 
\end{proof}

Suppose $A(K_n) = \langle [\mb] \rangle$. If $\#A(K_n) = 2$, then from Proposition \ref{tau as -1}, $\tau$ acts as inverse on $A(K_n)$. This is equivalent to $[\mb]^{\tau} = [\mb]$, because the order of $[\mb]$ is equal to 2. Otherwise, if $\#A(K_n) \geq 4$, then $\#A(K_n)^{\tau + 1} = \#A(K_n)/2$, and $A(K_n)$ is cyclic together imply that $A(K_n)^{\tau+1} = \langle ~ [\mb]^2 ~ \rangle$. Since every element of $A(K_n)^{\tau +1}$ is fixed by $\tau$, and $A(K_n)$ is a Galois module, we have $([\mb]^{\tau})^2 = ([\mb]^2)^{\tau} = [\mb]^2$. For a finite cyclic 2-group $G$, the map $x \mapsto x^2$ has exactly two elements in the kernel, namely the identity and the element of order 2. Therefore, if $([\mb]^{\tau})^2 = [\mb]^2$, then $[\mb]^{\tau} = [\mb]$ or $[\mb]^{\tau} = [\mb]\cdot[\ma]$, where $o([\ma]) = 2$.
 
\begin{rmk}\label{action of sigmatau}
We recall from Lemma \ref{bound of A(K_n)} that ${\rm{Gal}}(K_n/\Q_n(\sqrt{r})) = \langle \sigma \rangle$, and it acts as inverse on $A(K_n)$. From the above discussion, it is readily available that $[\mb]^{\sigma\tau} = ([\mb]^{-1})^{\tau} = ([\mb]^{\tau})^{-1} = [\mb]^{-1}$ or $[\mb]^{-1}\cdot[\ma]$. 
\end{rmk}

Since $K_n/k_n$ is a cyclic unramified extension of prime degree, by Hilbert's Theorem 94, there exists a nontrivial ideal class in $\mathcal{C}l_{k_n}$ that capitulates in $K_n$. Specifically, if $F/K$ is an unramified cyclic extension with lifting map $j$, then $\#Ker(j) = [F:K][ E(K): N_{F/K}(E(F))]$ (cf. \cite[Section 1.8]{lemmermeyer}). A method of identifying the Taussky condition for a cyclic unramified extension can be found in \cite[Corollary 1.8.3]{lemmermeyer}. If $F/K$ is an unramified cyclic extension of degree $d$ with ${\rm{Gal}}(F/K) = \langle \psi \rangle$, then for $\nu:= 1 + \psi + \cdots + \psi^{d-1}$, we consider the subgroups $C_{\nu}:= \{ [x] \in \mathcal{C}l_F: [x]^{\nu} = id \}$ and $C^{1-\psi}:= \{ [x]^{1-\psi} : [x] \in \mathcal{C}l_F\}$ of $\mathcal{C}l_F$. In this setting, it was proved that $[C_{\nu}: C^{1-\psi}] = \#\left( Ker(j)\cap N_{F/K}(\mathcal{C}l_F)\right)$. 

\smallskip

\noindent If $F/K$ is an unramified quadratic extension, then $\#Ker(j) = 2 \cdot [ E(K): N_{F/K}(E(F))]$.
The second factor $[ E(K): N_{F/K}(E(F))]$ can be seen in the second part of the genus formula in Theorem \ref{genusfor}. Thus, $\#Ker(j)$ will be a nontrivial power of 2, say $2^r$. As a result, $Ker(j)$ and $Ker(j)\cap N_{F/K}(\mathcal{C}l_F)$ must be contained in $A(K)$. Hence, it is sufficient to consider $Ker(j)\cap N_{F/K}(A(F))$ for extensions of degrees of powers of 2. Furthermore, there exists $s \geq 0$ such that $[C_{\nu}: C^{1-\psi}] = \left( \#Ker(j)\cap N_{F/K}(\mathcal{C}l_F)\right) = 2^s$. This implies that for every $[x] \in C_{\nu}$, there exists $[y] \in \mathcal{C}l_F$ such that $[x]^{2^s} = [y]^{1 - \psi}$. If $[x]$ belongs to the odd part of $\mathcal{C}l_F$ (that is, $[x] \not\in A(F)$), then there exists an odd $t \geq 0$ such that $[x]^t = id$. As $t$ and $2^s$ are relatively prime, there exist integers $x_0$ and $y_0$ such that $x_02^s + y_0t = 1$. Therefore, $[x]= [x]^{x_02^s + y_0t} = [x]^{x_02^s} = ([y]^{1 - \psi})^{x_0} =  ([y]^{x_0})^{1 - \psi}$, means that $[x]$ belongs to $C^{1-\psi}$. This demonstrates that the odd part of $C_{\nu}$ is contained in $C^{1-\psi}$. The reverse containment is clear as $C^{1-\psi} \subseteq C_{\nu}$. Therefore, the odd part of the subgroups $C_{\nu}$ and $C^{1-\psi}$ are equal. Hence, it is enough to consider the even part of the class group to evaluate $[C_{\nu}: C^{1-\psi}]$, that is, $[C_{\nu}: C^{1-\psi}] = [A(F)_{\nu}: A(F)^{1-\psi}] = \left(\#Ker(j)\cap N_{F/K}(A(F))\right)$. We now prove the following result, which will help us realize the action of $\tau$ on $A(K_n)$ from the possibilities.

\begin{lemma}\label{tau as 1}
For $n \geq 1$, ${\rm{Gal}}(K_n/\Q_n(\sqrt{p})) = \langle \tau \rangle$, $\tau$ acts as identity on $A(K_n)$. Hence, $\sigma\tau$ acts as the inverse on $A(K_n)$.    
\end{lemma}
\begin{proof}
Let $A(K_n) =\langle [\mb] \rangle$. By Proposition \ref{tau as -1}, if $\#A(K_n) = 2$, then $\tau$ acts as inverse. However, the order of $[\mb]$ is equal to 2 implies $[\mb] = [\mb]^{-1} = [\mb]^{\tau}$.

\noindent If $\#A(K_n) = 4$, then $[\mb]^{\tau} = [\mb]\cdot[\ma]$ is the same as $[\mb]^{\tau} = [\mb]^{-1}$, and $\#A(K_n)^{\tau + 1} = 1$. This is a contradiction to part 1 of Lemma \ref{bound of A(K_n)}, and thus, $[\mb]^{\tau} = [\mb]$. 

\noindent If $\#A(K_n) \geq 8$, then, ${\rm{Gal}}(L(K_n)/k_n)$ is a nonabelian group of order 16. Hence, it is obvious that ${\rm{Gal}}(L(K_n)/k_n) = {\rm{Gal}}(\tL(k_n)/k_n) \neq Q_8, \ (2,2)$. This rejects Cases 1 and 2 of Theorem \ref{Kisilevsky 1}. Suppose $[\mb]^{\tau} = [\mb]\cdot[\ma]$. Then by Remark \ref{action of sigmatau}, $[\mb]^{\sigma\tau} = [\mb]^{-1}\cdot[\ma]$. We consider the unramified quadratic extension $K_n/k_n$ and define $\nu = 1 + \sigma\tau$. For $[x] \in A(K_n)$, there exists an $s \geq 0$, such that $[x] = [\mb]^s$, and $[x]^{\sigma\tau} = ([\mb]^{s})^{\sigma\tau} = [\mb]^{-s}\cdot[\ma]^{s}$. Thus, $[x]^{\sigma\tau} = [x]^{-1}$ if $s$ is even and $[x]^{\sigma\tau} = [x]^{-1}\cdot[\ma]$ if $s$ is odd. Consequently, every even power of $[\mb]$ will belong to $A(K_n)_{\nu} := \{ [x] \in A(K_n): [x]^{\nu} = id \}$, and conversely. Therefore, $A(K_n)_{\nu} = A(K_n)^2 = A(K_n)^{\tau + 1}$. 

\noindent Next, $\sigma$ acts as the inverse implies $A(K_n)^{1 - \sigma\tau} = A(K_n)^{\tau + 1}$. Thus, we obtain $[A(K_n)_{\nu}: A(K_n)^{1-\sigma\tau}] = \#(Ker(j_3)\cap N_{K_n/k_n}(A(K_n))) = 1$, where $j_3$ is the lifting map from $A(k_n)$ to $A(K_n)$. This means that if $[\mb]^{\tau} = [\mb]\cdot[\ma]$, then $K_n/k_n$ satisfies Taussky Condition (B). Appealing to Theorem \ref{Kisilevsky 1}, we find that Case 3 (a) holds, and ${\rm{Gal}}(\tL(k_n)/k_n)$ is a semi-dihedral group of order at least 16. For this $n$, and the cyclotomic $\Z_2$-extension of the field $k_n$, we have $(k_n)_{\infty} = k_{\infty}$ and $X((k_n)_{\infty}) = X(k_{\infty}) \cong (2,2)$. Since ${\rm{Gal}}(\tL(k_n)/k_n)$ is semidihedral, from Theorem \ref{mouhib-mova 2008}, ${\rm{Gal}}(\tL((k_n)_{\infty})/(k_n)_{\infty}) = {\rm{Gal}}(\tL(k_{\infty})/k_{\infty})$ must be semidihedral. But this contradicts Theorem \ref{never sd}. Therefore, if $\#A(K_n) \geq 8$, $[\mb]^{\tau} = [\mb]$.
\end{proof}

\subsection*{Proof of Theorem \ref{Thm2}}
\begin{proof}
If $n =0$, then from Lemma \ref{rank A(K1)} and the paragraph before Proposition \ref{A(Q_n(p)}, $\#A(K_0) = \#A(\Q_0(\sqrt{p}))$. For any $n$, ${\rm{Gal}}(L(K_n)/K_n) \cong A(K_n)$ is an abelian extension, and ${\rm{Gal}}(K_n/\Q_n(\sqrt{p})) = \langle \tau \rangle$ is of degree 2. Since $K_n/\Q_n(\sqrt{p})$ is quadratic and $L(K_n)$ is the 2-Hilbert class field of $K_n$, $L(K_n)/\Q_n(\sqrt{p})$ is a Galois extension. Since $\tau$ acts as 1 on ${\rm{Gal}}(L(K_n)/K_n)$ from Lemma \ref{tau as 1}, $L(K_n)/\Q_n(\sqrt{p})$ is a finite abelian extension.

\noindent We now refer to Theorem \ref{GrasCap} for the ramified extension $K_n/\Q_n(\sqrt{p})$ for $n \geq 1$. We note that according to Theorem \ref{GrasCap}, for the extension $K_n/\Q_n(\sqrt{p})$, we have $\ell = 2$, $N = 1$, $\psi = \tau$, and $m(K_n) = 1$ (from Lemma \ref{tau as 1}). For these values, $s = 0$, $m = 1$, and $e = 1$. Therefore, if we restate Theorem \ref{GrasCap} for this extension, we obtain that $[x] \in A(\Q_n(\sqrt{p}))$ capitulates in $A(K_n)$ if there exists $[y] \in A(K_n)$ of order 2 such that $N_{K_n/\Q_n(\sqrt{p})}([y]) = [x]$. 

\noindent Suppose $\#A(K_n) = \#A(\Q_n(\sqrt{p}))$. Then, since $N_{K_n/\Q_n(\sqrt{p})}$ is surjective, it must be an isomorphism. As $n \geq 1$, $A(\Q_n(\sqrt{p}))$ contains an element of order 2, say $[x]$. Let $[y] \in A(K_n)$ be the pre-image of $[x]$ under $N_{K_n/\Q_n(\sqrt{p})}$. Then, order of $[y]$ must be equal to 2. From the previous paragraph, it is clear that $[x]$ capitulates in $A(K_n)$.

\noindent Conversely, suppose there is capitulation of some non-trivial class of $A(\Q_n(\sqrt{p}))$ in $A(K_n)$. Then, the kernel of the lifting map $j$ must be non-trivial with some $[x] \in A(\Q_n(\sqrt{p}))$ of order 2 that capitulates in $A(K_n)$. Let $[y] \in A(K_n)$ be a class such that $N_{K_n/\Q_n(\sqrt{p})}([y]) = [x]$. From Section 1 of \cite{gras3}, we observe that $(j \circ N_{K_n/\Q_n(\sqrt{p})} )[y] = [y]^{1 + \tau}$. This implies, $j[x] = id = [y]^2$ due to Lemma \ref{tau as 1}. Therefore, the order of $[y]$ is at most 2. But, as $[x]$ is non-trivial, order of $[y]$ must be precisely equal to 2. This further means that the unique element of order 2 of $A(K_n)$ (as $A(K_n)$  is cyclic) maps to a non-trivial element of $A(\Q_n(\sqrt{p}))$. From part 2 of Lemma \ref{bound of A(K_n)}, it is immediate that the kernel of norm map can have at most 2 elements. Thus, in this case, the kernel of norm map is trivial. From this, we realize that the norm map $N_{K_n/\Q_n(\sqrt{p})}$ is an isomorphism from $A(K_n)$ to $A(\Q_n(\sqrt{p}))$, and hence, the order of these two groups must be equal.    
\end{proof}

\subsection*{Proof of Corollary \ref{Cor3}}

\begin{proof}
The field $L(K_n)$ is unramified over $K_n$, and $K_n$ is ramified over $\Q_n(\sqrt{p})$ at primes above 2 and $r$. Hence, $L(K_n)$ is also ramified over $\Q_n(\sqrt{p})$ at primes above 2 and $r$, each with ramification index 2.  Therefore, $L(K_n)$ must be contained in the maximal abelian extension of $\Q_n(\sqrt{p})$ unramified outside primes above 2 and $r$. The field $L(k_n) = K_n(\sqrt{p_1}) = \Q_n(\sqrt{p}, \sqrt{r}, \sqrt{p_1})$ is a biquadratic extension of $\Q_n(\sqrt{p})$. Now, ${\rm{Gal}}(L(K_n)/\Q_n(\sqrt{p}))$ contains a biquadratic quotient namely, ${\rm{Gal}}(L(k_n)/\Q_n(\sqrt{p}))$. Thus we obtain that ${\rm{Gal}}(L(K_n)/\Q_n(\sqrt{p}))$ must have rank at least 2, and hence the result.     
\end{proof}

\section{Order of $A(K_1)$}
\noindent  We first identify the behaviour of the ideal $r\mathcal{O}_{\Q_1}$ in certain extensions of $\Q_1$ which will be essential in the proof of Theorem \ref{Thm3}.

\begin{lemma}\label{r is inert}
Let $r$ and $p$ be prime numbers satisfying Condition 1 and $p_1, p_2$ be totally positive prime elements in $\Q_1$ such that $p = p_1 \cdot p_2$ in $\Q_1$. Then, the following hold:
\begin{enumerate}
    \item If $r \equiv 3 \Mod{8}$, then the prime ideal $r\mathcal{O}_{\Q_1}$ is inert in the extension $\Q_1(\sqrt{p_i})/\Q_1$, for $i = 1,2$.
    \item If $r \equiv 7 \Mod{8}$ and $r\mathcal{O}_{\Q_1} = \langle r_1 \rangle \cdot \langle r_2 \rangle$ in $\Q_1$, then exactly one of $\langle r_1 \rangle$ and $\langle r_2 \rangle$ is inert in $\Q(\sqrt{p_1})/\Q_1$ (and similarly in $\Q(\sqrt{p_2})/\Q_1$). Also, if $\langle r_1 \rangle$ is inert in $\Q_1(\sqrt{p_1})$, then it must split in $\Q_1(\sqrt{p_2})$.
\end{enumerate}
\end{lemma}
\begin{proof}
(1). Let $r \equiv 3 \Mod{8}$. Since $p \equiv 1 \Mod{4}$, from Condition 1, $\left( \dfrac{r}{p} \right) = -1$. This indicates that $r$ is not a square modulo $p$. Since $p \equiv 1 \Mod{8}$, the ideal $p\mathcal{O}_{\Q_1}$ can be expressed as $\langle p_1 \rangle \cdot\langle p_2 \rangle$, with $\mathcal{O}_{\Q_1}/\langle p_i \rangle \cong \Z/p\Z$ (for $i = 1,2$). Therefore, $r$ is not a square modulo $p_i$ in $\mathcal{O}_{\Q_1}$. Using the generalization of Legendre symbols in number fields, we obtain $\left( \dfrac{r}{p_i} \right) = -1$. The extension $\Q_1(\sqrt{p_i})/\Q_1$ is quadratic, with the minimal polynomial of $\sqrt{p_i}$ being $f(x) = x^2 - p_i$. Since $r \equiv 3 \Mod{8}$, the ideal $r\mathcal{O}_{\Q_1}$ is prime in $\Q_1$. Now, whether $r\mathcal{O}_{\Q_1}$ splits or remains inert in $\Q_1(\sqrt{p_i})$ depends on the factorization of $f$ modulo $r\mathcal{O}_{\Q_1}$ in $\mathcal{O}_{\Q_1}$. If $p_i$ is not a square modulo $r\mathcal{O}_{\Q_1}$, then $r\mathcal{O}_{\Q_1}$ is inert in $\Q_1(\sqrt{p_i})/\Q_1$. Thus, we need to find the value of $\left( \dfrac{p_i}{r} \right)$.

\noindent We produce a proof for $p_1$, as a similar proof would hold for $p_2$ as well. Our aim is to use Hecke's quadratic reciprocity stated in Theorem \ref{Reciprocity}. Since $\Q_1/\Q$ is a real quadratic extension, $R =2$, $x = p_1, y = r$, $(x_1,y_1) = (p_1, r)$, and $(x_2, y_2) = (p_2,r)$. As $p_i$'s are totally positive and $r >0$, $s(x_1,y_1) = s(x_2, y_2) =0$. Also, in the proof of Theorem 2 stated in \cite{nishino}, it has been proved that modulo 4, $p_i$'s are congruent to 1, or $3 + 2\sqrt{2} = (1 + \sqrt{2})^2$.
Therefore, from Theorem \ref{Reciprocity}, we have $\left( \dfrac{p_1}{r} \right) = (-1)^0 \cdot (-1)^{0} \cdot (-1) = -1$, and $p_1$ is not a square modulo $r\mathcal{O}_{\Q_1}$ in $\Q_1$. Hence, $r\mathcal{O}_{\Q_1}$ is inert in $\Q_1(\sqrt{p_1})/\Q_1$.

\smallskip

\noindent (2). Let $r \equiv 7 \Mod{8}$. Then, there exits $r_1, r_2 \in \Q_1$ such that $r\mathcal{O}_{\Q_1}$ can be factorized as $r\mathcal{O}_{\Q_1} = \langle r_1 \rangle \cdot \langle r_2 \rangle$. In this case, $\Z/r\Z \cong \mathcal{O}_{\Q_1}/\langle r_i \rangle$ for $i = 1,2$ and $\left( \dfrac{p}{r} \right) = \left( \dfrac{p}{r_i} \right) = \left( \dfrac{p_1p_2}{r_i} \right) = -1$. Hence, for each $i$, exactly one of $\left( \dfrac{p_1}{r_i} \right)$ and $\left(\dfrac{p_2}{r_i} \right)$ must be equal to $-1$, and the other symbol must be equal to 1. We furnish a proof for $i = 1$ as the other case can be dealt with similarly. Without loss of generality, let $\left( \dfrac{p_1}{r_1} \right) = -1$ (this implies that $\langle r_1 \rangle$ is inert in $\Q_1(\sqrt{p_1})/\Q_1$). The conjugates of $r_1$ are $r_1$ and $r_2$, and the conjugates of $p_1$ are $p_1, p_2$. Applying Theorem \ref{Reciprocity} on $p_1$ and $r_1$, we obtain that $\left( \dfrac{r_1}{p_1} \right) = -1$. Further, $\left( \dfrac{r}{p} \right) = \left( \dfrac{r_1r_2}{p_1} \right) = -1$, $\left( \dfrac{r_1}{p_1} \right) = -1$, and Theorem \ref{Reciprocity} all together imply that $\left( \dfrac{r_2}{p_1} \right) = \left( \dfrac{p_1}{r_2} \right) = 1$. This means that if $\langle r_1 \rangle$ is inert in $\Q_1(\sqrt{p_1})$, then $\langle r_2 \rangle$ must be a split prime in $\Q_1(\sqrt{p_1})$. Also, $\left( \dfrac{p}{r} \right) = -1$ implies that definitely, one of the ideals $\langle r_1 \rangle$ and $\langle r_2 \rangle$ must be inert in $\Q_1(\sqrt{p_1})$.
\end{proof}

The field $\Q_1( \sqrt{p}, \sqrt{p_1})$ is the 2-Hilbert class field of $\Q_1(\sqrt{p})$, with subfields $\Q_1(\sqrt{p_1})$ and $\Q_1(\sqrt{p_2})$ above $\Q_1$. As $p_1$ and $p_2$ are totally positive in $\Q_1$, from Theorem \ref{Reciprocity}, $\left( \dfrac{p_1}{p_2} \right) = \left( \dfrac{p_2}{p_1} \right)$. Therefore, the prime ideal $\langle p_1 \rangle$ remains inert in $\Q_1(\sqrt{p_2})/\Q_1$ if and only if $\langle p_2 \rangle$ remains inert in $\Q_1(\sqrt{p_1})/\Q_1$. The next lemma is a consequence of this assumption.

\smallskip

\begin{lemma}\label{rank T2}
Let $r \equiv 3 \Mod{8}$, and suppose that the prime ideal $\langle p_1 \rangle$ of $\Q_1$ is inert in $\Q_1(\sqrt{p_2})/\Q_1$. If $T_2:= \Q_1(\sqrt{rp_1}, \sqrt{p_2})$, then ${\rm{rank}}_2(A(T_2)) \leq 1$.    
\end{lemma}
\begin{proof}
For simplicity, we shall use $F$ to denote the field $\Q_1(\sqrt{p_2})$ in this proof. Since $\#A(\Q_1(\sqrt{p})) = 2$, by Burnside's basis theorem, $\#A(\Q_1( \sqrt{p}, \sqrt{p_1})) = 1$. 
Now, $p_1\mathcal{O}_F$ is ramified in $\Q_1( \sqrt{p}, \sqrt{p_1})/F$. Hence, the norm map between class groups of these fields must be surjective and thus $\#A(F) = 1$. We note that the only prime ramified in $F/\Q_1$ is $\langle p_2 \rangle$. So,
if we apply the second form of the genus formula for this extension, then we obtain $1 = \#A(\Q_1) \times \dfrac{2^{1-1}}{\left[ E(\Q_1): N_{F/\Q_1}(E(F)) \right]}$. Therefore, every unit of $\Z[\sqrt{2}] = \mathcal{O}_{\Q_1}$ is the norm of some unit of $\mathcal{O}_{F}$. In particular, there exists $u \in \mathcal{O}_{F}$ such that its norm over $\Q_1$ is $1 + \sqrt{2}$.\\

\noindent It follows from our assumption, part 1 of Lemma \ref{r is inert}, and \cite[Proposition 4.5]{L-S_arx} that the primes of $F$ ramified in $T_2$ are $p_1\mathcal{O}_F$, $r\mathcal{O}_F$, and $\ell$, respectively, where $\ell$ is the unique prime ideal above 2. Consequently, from genus formula for $T_2/F$, $\ 2^{\ {\rm{rank}}_2A(T_2)} = \dfrac{ 2^{3-1}}{\left[E(F):E(F)\cap N_{T_2/F}(T_2^{\times})\right]} \leq 2^2$. Suppose ${\left[E(F):E(F)\cap N_{T_2/F}(T_2^{\times})\right]} = 1$. Then, there exists $\alpha \in T_2^{\times}$, such that its norm over $F$ is equal to $u$, where $u$ is defined in the previous paragraph. In that case, $N_{T_2/\Q_1}(\alpha) = N_{F/\Q_1}(u) = 1+\sqrt{2}$. Moreover, $N_{T_2/\Q_1}(\alpha) = N_{k_1/\Q_1}(N_{T_2/k_1}(\alpha)) = 1+\sqrt{2}$, as $T_2 = \Q_1(\sqrt{rp_1}, \sqrt{p_2}) = \Q(\sqrt{rp}, \sqrt{p_2}) = k_1(\sqrt{p_2})$. But since $r \equiv 3 \Mod{8}$, this is a contradiction to Proposition 3.3 of \cite{L-S_arx} which states that $1+\sqrt{2}$ is not a norm in the extension $\Q_1(\sqrt{d})/\Q_1$ if $d$ has a prime factor congruent to $3 \Mod{4}$. Thus, ${\left[E(F):E(F)\cap N_{T_2/F}(T_2^{\times})\right]} \geq 2$, and ${\rm{rank}}_2(A(T_2)) \leq 1$. 
\end{proof}

Let $a,b \in \Z_{\geq 0}$ be such that $p_1 = a + b\sqrt{2}$ and $p_2 = a - b\sqrt{2}$, so that $p = a^2 - 2b^2$. This means that $p_1 \equiv 2a \Mod{p_2}$. As $\mathcal{O}_{\Q_1}/\langle p_i \rangle \cong \Z/p\Z$, 2 is a square modulo $p$ implies that 2 is a square modulo $p_2$ in $\Q_1$. Therefore, $\left( \dfrac{p_1}{p_2} \right) = -1$ if and only if $\left( \dfrac{a}{p_2} \right) = -1$, which is equivalent to $\left( \dfrac{a}{p} \right) = -1$. Also, $\left( \dfrac{2}{p} \right)_4 = -1$ means that 2 is not a fourth power modulo $p$. Hence, the square root of 2 in modulo $p$ is not a square itself. Now, $a^2 \equiv 2b^2 \Mod{p}$ implies $a \equiv \pm tb \Mod{p}$, where $t^2 \equiv 2 \Mod{p}$. Since $-1$ is also a square modulo $p$, $\left( \dfrac{a}{p} \right) = -1$ is equivalent to $\left( \dfrac{b}{p} \right) = 1$. 

\subsection*{Proof of Theorem \ref{Thm3}}

\begin{proof}
 From the discussion in the previous paragraph, it is evident that if $p = a^2 - 2b^2$, then, $\left( \dfrac{a}{p} \right) = -1$ if and only if the prime ideal $\langle p_1 \rangle$ of $\Q_1$ is inert in $\Q_1(\sqrt{p_2})/\Q_1$. From Lemma \ref{bound of A(K_n)}, it is direct that $2 \leq \#A(K_1) \leq 4$. We suppose that $\#A(K_1) = 4$. From Theorem \ref{Thm2}, $L(K_1)$ is an abelian extension of $\Q_1(\sqrt{p})$ of degree 8. We observe that $\Q_1(\sqrt{p}) \subset K_1 \subset L(K_1)$, where $L(K_1)/K_1$ is cyclic of degree 4. On the other hand, $\Q_1(\sqrt{p},\sqrt{r}, \sqrt{p_1})/\Q_1(\sqrt{p}) \cong (2,2)$ is a subextension of $L(K_1)/\Q_1(\sqrt{p})$. Putting these together, we infer that $L(K_1)/\Q_1(\sqrt{p}) \cong (2,4)$.

\noindent A consequence of $\#A(K_1) = 4$ is that $L(k_1) \subsetneq L(K_1)$. We have the field extension $k_1 \subset K_1 \subset L(k_1) \subset L(K_1) \subset L(L(k_1)$, and by Burnside's basis theorem, $L(L(k_1)) \supset L(K_1) = \tL(K_1) = \tL(k_1) \supset L(L(k_1))$ $L(K_1)$. Therefore, $L(K_1)$ can also be viewed as the 2-Hilbert class field of $L(k_1)$. We now refer to the extension $\Q_1(\sqrt{p}) \subset \Q_1(\sqrt{p}, \sqrt{p_1}) \subset L(k_1) \subset L(K_1)$. Since $L(K_1)$ is the 2-Hilbert class field of $L(k_1)$, it must be a Galois extension of degree 4 over $\Q_1(\sqrt{p_1}, \sqrt{p})$. Then, $L(k_1)/\Q_1(\sqrt{p_1}, \sqrt{p})$ is a subextension of $L(K_1)/\Q_1(\sqrt{p}, \sqrt{p_1})$ ramified at primes above 2 and $r$, and $L(K_1)/L(k_1)$ is unramified. Summing these reasons up, by class field theory, $L(K_1)/\Q_1(\sqrt{p}, \sqrt{p_1})$ cannot be a cyclic extension of degree 4. Hence, $L(K_1)/\Q_1(\sqrt{p_1}, \sqrt{p})$ has to be an extension of type $(2,2)$. Thus, there exist fields $F_1, F_2$ different from $L(k_1)$ such that $\Q_1(\sqrt{p_1}, \sqrt{p}) \subset F_1, F_2 \subset L(K_1)$. Likewise, as $L(k_1)/\Q_1(\sqrt{r}, \sqrt{p_1})$ is a quadratic extension ramified at the prime(s) above $p_2$, $L(K_1)/ \Q_1(\sqrt{r}, \sqrt{p_1})$ is of the form $(2,2)$. Therefore, there exist intermediate fields $H_1$ and $H_2$ distinct from $L(k_1)$ such that $ \Q_1(\sqrt{r}, \sqrt{p_1})\subset H_1, H_2 \subset L(K_1)$. Throughout this proof, the fields $F_1, F_2, H_1$, and $H_2$ will be fixed.  

\noindent Our assumption $\#A(K_1) = 4$ also yields that $L(K_1)/k_1$ is an unramified, non-abelian extension of degree 8. As a result, $L(K_1)/\Q_1$ is a Galois, non-abelian extension of degree 16. From the previous paragraphs, we infer that ${\rm{Gal}}(L(K_1)/\Q_1)$ has at least two subgroups of type $(2,2)$, (namely, ${\rm{Gal}}(L(K_1)/\Q_1(\sqrt{p}, \sqrt{p_1}))$ and ${\rm{Gal}}(L(K_1)/\Q_1(\sqrt{r}, \sqrt{p_1}))$), 
a subgroup of type $(2,4)$, (corresponding to ${\rm{Gal}}(L(K_1)/\Q_1(\sqrt{p}))$, 
and a quotient of the form $(2,2,2)$ (with respect to the subextension $L(k_1) = \Q_1(\sqrt{p}, \sqrt{r}, \sqrt{p_1})/\Q_1$). Moreover, we have at least 7 subextensions of degree 2 over $\Q_1$ contained in $L(K_1)$. These are $\Q_1(\sqrt{p})$, $\Q_1(\sqrt{r})$, $k_1= \Q_1(\sqrt{pr})$, $\Q_1(\sqrt{p_1})$, $\Q_1(\sqrt{p_2})$, $\Q_1(\sqrt{rp_1})$, and $\Q_1(\sqrt{rp_2})$. Therefore, ${\rm{Gal}}(L(K_1)/\Q_1)$ has at least 7 quotients of order 2. None of the groups mentioned in Remark \ref{groups 16} can be isomorphic to ${\rm{Gal}}(L(K_1)/\Q_1)$ as their properties stated therein do not match with our current observations. 
The groups of order 16 that are remaining to be pondered upon are $Q_8 \oplus \Z/2\Z, D_8 \curlyvee \Z/4\Z$ (the central product of $D_8$ and $\Z/4\Z$), and $D_8 \oplus \Z/2\Z$. Let $\mG = {\rm{Gal}}(L(K_1))/\Q_1$. We will deal with the possibility of $\mG$ being equal to each of the aforementioned groups. 

\noindent{\bf Case 1}. $\mG = Q_8 \oplus \Z/2\Z$: In this case, $\mG$ has exactly three subgroups of order 2, all of which are normal. Thus, by Galois correspondence, there exist fields $M_i$ ($i =1,2,3$) such that for each $i$, $\Q_1 \subset M_i \subset L(K_1)$, $[L(K_1):M_i] = 2$, and $M_i/\Q_1$ is Galois. One of these is $L(k_1)$, and we fix $M_3 = L(k_1)$. We already have five fields, namely $F_1, F_2, H_1, H_2$, and $L(k_1)$ above which $L(K_1)$ is quadratic. Since $\mG$ has only three subgroups of order 2, some these fields must be equal. Clearly, $F_1 \neq F_2 \neq L(k_1)$, and $H_1 \neq H_2 \neq L(k_1)$. Therefore, without loss of generality, $M_1 = F_1 = H_1$ and  $M_2 = F_2 = H_2$. By definition, $\Q_1(\sqrt{p_1}, \sqrt{p}) \subset F_1$ and $\Q_1(\sqrt{r}, \sqrt{p_1})\subset H_1$. That way, $M_1$ must contain $\Q_1(\sqrt{p_1}, \sqrt{p}, \sqrt{r}) = M_3$ which is not possible. A similar contradiction occurs with $M_2$ as well. Therefore, $\mG \neq Q_8 \oplus \Z/2\Z$.  

\noindent{\bf Case 2}. $\mG =  D_8 \curlyvee \Z/4\Z$: The group $\mG$ has only one normal subgroup of order 2. Since $L(k_1)/\Q_1$ is Galois, we note that $F_1$ and $F_2$ cannot be Galois over $\Q_1$. Thus, $F_1$ and $F_2$ are isomorphic field extensions over $\Q_1$ (as their compositum $L(K_1)$ is Galois over $\Q_1$) with isomorphism $\theta$ such that $\theta$ is also an automorphism of $\Q_1(\sqrt{p}, \sqrt{p_1})$ that fixes $\Q_1$. We shall break this case into two subcases depending on modulo 8 conditions on $r$.

\noindent (1). When $r \equiv 3 \Mod{8}$: Since $p \equiv 1 \Mod{8}$ and $\left( \dfrac{p}{r} \right) = -1$, the prime ideals $\sqrt{2}\mathcal{O}_{\Q_1}$ and $r\mathcal{O}_{\Q_1}$ split in $\Q_1(\sqrt{p})$, where they can be factorized as $\ell_1\cdot\ell_2 $ and $\mr_1\cdot \mr_2$, respectively. Next, from \cite[Proposition 4.5]{L-S_arx} and Lemma \ref{r is inert}, we gather that for $i=1,2$, $\ell_i$ and $\mr_i$ remain inert in $\Q_1(\sqrt{p_1}, \sqrt{p})$. We denote these primes ideals in $\Q_1(\sqrt{p_1}, \sqrt{p})$ as $\overline{\ell_i}$ and $\overline{\mr_i}$ for $i = 1,2$. Given that $\#A(\Q_1(\sqrt{p}, \sqrt{p_1})) = 1$, the extensions $F_i/\Q_1(\sqrt{p}, \sqrt{p_1})$ must be ramified at some prime. Since the primes above 2 and $r$ are the only ones that are ramified, without loss of generality, we may assume that $\overline{\mr_1}$ is ramified in $F_1/\Q_1(\sqrt{p}, \sqrt{p_1})$. Let $\widetilde{\mr_1}$ be the prime above $\overline{\mr_1}$ in $F_1/\Q_1(\sqrt{p}, \sqrt{p_1})$. Then, $\theta(\widetilde{\mr_1})$ will be the prime above $\theta(\overline{\mr_1})$ in $F_2/\Q_1(\sqrt{p}, \sqrt{p_1})$. Here, $\theta(\overline{\mr_1})$ is the conjugate of $\overline{\mr_1}$ in $\Q_1(\sqrt{p}, \sqrt{p_1})$. But, $\mr_1$ is inert implies $\theta(\overline{\mr_1}) = \overline{\mr_1}$ in $\Q_1(\sqrt{p}, \sqrt{p_1})$. Therefore, $\theta(\widetilde{\mr_1})$ will be the prime above $\overline{\mr_1}$ in $F_2/\Q_1(\sqrt{p}, \sqrt{p_1})$. Since $F_1$ and $F_2$ are isomorphic, this implies that $\overline{\mr_1}$ is ramified in $F_2$ as well. Therefore, $\overline{\mr_1}$ is ramified in $F_1,F_2$, and $L(k_1)$, and consequently, it must be totally ramified in $L(K_1)/\Q_1(\sqrt{p_1}, \sqrt{p})$. This is not possible as $L(K_1)/L(k_1)$ is an unramified extension. On the similar lines, we can also show that the primes $\overline{\mr_2}, \overline{\ell_1}$, and $\overline{\ell_2}$ cannot be ramified in $F_1/\Q_1(\sqrt{p}, \sqrt{p_1})$, and this is a contradiction. 

\noindent(2). When $r \equiv 7 \Mod{8}$: For $i = 1,2$, $L(K_1)/\Q_1(\sqrt{p_i})$ is a Galois extension of degree 8. The subgroups of order 8 of $\mG$ can be $D_8, Q_8$, or $(2,4)$. We have $\Q_1(\sqrt{p_i}) \subseteq \Q_1(\sqrt{p_i}, \sqrt{r}), ~\Q_1(\sqrt{p_i}, \sqrt{p}) \subset L(K_1)$, where both ${\rm{Gal}}(L(K_1)/\Q_1(\sqrt{p_i}, \sqrt{r}))$ and ${\rm{Gal}}(L(K_1)/\Q_1(\sqrt{p_i}, \sqrt{p}))$ are of the type $(2,2)$. Consequently, ${\rm{Gal}}(L(K_1)/ \Q_1(\sqrt{p_i})) \cong D_8$ for $i = 1,2$, with $F_1$ and $F_2$ being isomorphic, non-Galois extensions of $\Q_1(\sqrt{p_i})$. 
We recall $r_1$ and $r_2$ defined in part (2) of Lemma \ref{r is inert} and fix $\left( \dfrac{p_1}{r_1} \right) = -1$ without loss of generality. Then, from part (2) of Lemma \ref{r is inert}, $r_1\mathcal{O}_{\Q_1}$ is inert in $\Q_1(\sqrt{p_1})/\Q_1$, and $r_2\mathcal{O}_{\Q_1}$ splits in $\Q_1(\sqrt{p_1})/\Q_1$. Let $\widetilde{r_1}$ be the prime ideal above $r_1\mathcal{O}_{\Q_1}$, and $r_2\mathcal{O}_{\Q_1(\sqrt{p_1})} = r_{21}\cdot r_{22}$ in $\Q_1(\sqrt{p_1})/\Q_1$. As $\Q_1(\sqrt{p_1}, \sqrt{p})/\Q_1$ is a biquadratic extension, we deduce that in the extension $\Q_1(\sqrt{p_1}, \sqrt{p})/\Q_1(\sqrt{p_1})$, we have the factorizations: $\widetilde{r_1} = \overline{r_{11}}\cdot \overline{r_{12}}$, $r_{21} = \overline{r_{21}}$, and $r_{22} = \overline{r_{22}}$. For each $i, j \in \{ 1,2\}$, $\overline{r_{ij}}$ is ramified in the extension $L(k_1)/\Q_1(\sqrt{p_1}, \sqrt{p})$. Since $L(K_1)/L(k_1)$ is an unramified extension, each $\overline{r_{ij}}$ has ramification index 2 in the biquadratic extension $L(K_1)/\Q_1(\sqrt{p_1}, \sqrt{p})$. If $F_2$ is the inertia field of $\overline{r_{21}}$, then $\overline{r_{21}}$ must be ramified in $F_1/\Q_1(\sqrt{p_1}, \sqrt{p})$. Following the arguments presented in the last part of the subcase $r \equiv 3 \Mod{8}$, we obtain that $\overline{r_{21}}$ is ramified in $F_2/\Q_1(\sqrt{p_1}, \sqrt{p})$, which is a contradiction. Therefore, $\mG \neq D_8 \curlyvee \Z/4\Z$.

\noindent{\bf Case 3}. $\mG =  D_8 \oplus \Z/2\Z$: Again, we break this case into subcases depending on the modulo 8 conditions on $r$.

\noindent (1). When $r \equiv 3 \Mod{8}$: The subgroups of order 8 of $\mG$ are of the type $(2,2,2), (2,4)$, and $D_8$. As $\#A(K_1) = 4$, in this case, $L(K_1)/k_1$ must be nonabelian and the Galois group must be isomorphic to $D_8$. We consider the extension $L(K_1)/k_1$. The group $D_8$ has exactly one cyclic subgroup of order 4, and here it corresponds to the extension $L(K_1)/K_1$. The quadratic subextensions of $k_1$ other than $K_1$ are $T_1= \Q_1(\sqrt{pr}, \sqrt{p_1})$, and $T_2 = \Q_1(\sqrt{pr}, \sqrt{p_2}) = \Q_1(\sqrt{p_1r}, \sqrt{p_2})$. From the structure of subgroups of $D_8$, and due to Galois correspondence, $L(K_1)/T_i$ ($i = 1,2$) is an unramified extension of type $(2,2)$. This implies ${\rm{rank}}_2A(T_i) \geq 2$. But this is a contradiction to Lemma \ref{rank T2} due to our initial assumption that $p = a^2 - 2b^2$ with $\left( \dfrac{a}{p} \right) = -1$. 

\noindent (2). When $r \equiv 7 \Mod{8}$: The group $D_8 \oplus \Z/2\Z$ has four subgroups isomorphic to $D_8$, two of the form $(2,2,2)$ and one subgroup of the type $(2,4)$. We have already shown that ${\rm{Gal}}(L(K_1)/\Q_1(\sqrt{p})) \cong (2,4)$ and from subcase (2) of the case $\mG = D_8 \curlyvee \Z/4\Z$, we note that ${\rm{Gal}}(L(K_1)/\Q_1(\sqrt{p_i})) \cong (2,2,2)$. Thus, ${\rm{Gal}}(L(K_1)/\Q_1(\sqrt{rp_i})) \cong D_8$ for $i = 1,2$. From Condition (\ref{Cond}), the ideal $p_2\mathcal{O}_{\Q_1}$ is inert in $\Q_1(\sqrt{r})/\Q_1$. Due to the discussion before Lemma \ref{rank T2}, $\left(\dfrac{a}{p}\right) = -1$ implies that $p_2\mathcal{O}_{\Q_1}$ is inert in $\Q_1(\sqrt{p_1})/\Q_1$ as well. As $\Q_1(\sqrt{r}, \sqrt{p_1})/\Q_1$ is a biquadratic extension, $p_2\mathcal{O}_{\Q_1}$ splits as the product $p_{21}\cdot p_{22}$ in $\Q_1(\sqrt{rp_1})/\Q_1$, and the ideals $p_{21}$ and $p_{22}$ remain inert in $\Q_1(\sqrt{r}, \sqrt{p_1})/\Q_1(\sqrt{rp_1})$. Let $\overline{p_{21}}$ and $\overline{p_{22}}$ be the prime ideals of $\Q_1(\sqrt{r}, \sqrt{p_1})$ above $p_{21}$ and $p_{22}$, respectively. As ${\rm{Gal}}(L(K_1)/\Q_1(\sqrt{rp_1})) \cong D_8$, the fields $H_1$ and $H_2$ defined earlier in this proof must be isomorphic, non-Galois extensions over $\Q_1(\sqrt{rp_1})$. Suppose that $H_2$ is the inertia field of $\overline{p_{21}}$ in the extension $L(K_1)/\Q_1(\sqrt{r}, \sqrt{p_1})$. With reasons similar to those discussed in both the subcases of Case 2, we obtain a contradiction and that way, $\mG \neq D_8 \oplus \Z/2\Z$.

\noindent We find that $\mG$ can be none of $Q_8 \oplus \Z/2\Z, D_8 \curlyvee \Z/4\Z$, and $D_8 \oplus \Z/2\Z$. Hence, we conclude that our assumption was wrong and $\#A(K_1) = 2$ if $p = a^2 - 2b^2$ with $\left( \dfrac{a}{p} \right) = -1$.
\end{proof}

\section{Concluding Remarks}
\noindent In Lemma \ref{A_(k0) = 2}, we saw how the decomposition of certain prime ideals has an effect on the order of the 2-class group. Based on a similar idea, we propose the following result:
\begin{propn}\label{2 is principal}
    Suppose $\ell_1$ and $\ell_2$ are prime ideals in $\Q(\sqrt{p})$ such that $2\mathcal{O}_{\Q(\sqrt{p})} = \ell_1\cdot \ell_2$. If $\ell_1$ and $\ell_2$ are principal, then $\#A(K_1) = 2$.
\end{propn}
\begin{proof}
    Since $r \equiv 3 \Mod{4}$ and $p \equiv 1 \Mod{8}$, $2\Z$ is ramified in $k/\Q$ and splits in $\Q(\sqrt{p})/\Q$. Let $\mathfrak{l}$ be the prime ideal in $k$ above $2\Z$. Then, $\mathfrak{l}$ splits in $K/k$, and hence must be principal in $k$ as $K$ is the 2-Hilbert class field of $k$. Thus, there exists $\ell \in k$ such that $\mathfrak{l} = \ell\mathcal{O}_k$, and $(\ell\mathcal{O}_k)^2 = 2\Z$. Let $\ell_1$ and $\ell_2$ be the prime ideals of $\Q(\sqrt{p})$ such that $2\mathcal{O}_{\Q(\sqrt{p})} = \ell_1\cdot \ell_2$. In this scenario, there exist prime ideals $L_1, L_2$ in $K$ such that $L_i ^2 = \ell_i\mathcal{O}_K$ for $i = 1,2$, and $L_1\cdot L_2 = \ell\mathcal{O}_K$. From our assumption, $[L_1]^2 = [L_2]^2 = id$. Therefore, by Lemma \ref{rank A(K1)}, $[L_i] \in A(K) = \{ id \}$, which further implies that $L_1$ and $L_2$ are principal in $K$. Since $K_1/K$ is ramified at primes above 2, there exist ideals $\widetilde{L_i}$ such that $(\widetilde{L_i})^2 = L_i\mathcal{O}_{K_1}$ for $i = 1,2$. Furthermore, the order of $[\widetilde{L_i}]$ is at most 2 as $L_i$ is principal for each $i$. If $\tilde{\mathfrak{l}}$ is the prime above $\mathfrak{l}$ in $k_1/k$, then from the proofs of Theorem 2.2 of \cite{fukuda-komatsu} and Theorem 2 of \cite{nishino}, $\tilde{\mathfrak{l}}$ remains inert in $k_1(\sqrt{p_1})/k_1$. Thus, the ideals $\widetilde{L_1}$ and $\widetilde{L_2}$ remain inert in the unramified extension $K_1(\sqrt{p_1})/K_1$, and thus are not principal in $K_1$. Applying the arguments presented in Lemma \ref{A_(k0) = 2} on the extension $K_1(\sqrt{p_1})/K_1$, we conclude $\#A(K_1) = 2$. 
\end{proof}

Theorem \ref{Thm3} and Proposition \ref{2 is principal} highlight the importance of finding and studying solutions to certain diophantine equations. The prime ideal $\ell_1$ is principal in $\Q(\sqrt{p})$ if and only if there exists a solution in integers to the equation $x^2 - py^2 = \pm 8$. For primes $p \equiv 9 \Mod{16}$ and $\left(\dfrac{2}{p}\right)_4 = -1$ such that $p \leq 10,000$, using PariGP, it can be seen that the prime ideals $\ell_i$ are principal in $\Q(\sqrt{p})$ except for the values $p =$ 761, 1129, 2153, 2713, 2777, 4441, 4649, 4729, 4889, 5273, 5417, 7673, 9049, 9833.  

\noindent It can also be verified that for all $p \leq 10,000$ of the form $p \equiv  9 \Mod{16}$ and $\left( \dfrac{2}{p} \right)_4 = -1$, if $a$ and $b$ are integers such that $a^2 - 2b^2 = p$, then indeed $\left( \dfrac{a}{p} \right) = -1$. This observation leads to the following question:

\noindent{\bf Question}: Let $p$ be a prime such that $p \equiv 9 \Mod{16}$ and $\left(\dfrac{2}{p}\right)_4 = -1$. If $a$ and $b$ are integers such that $a^2 - 2b^2 = p$, is it always true that $\left( \dfrac{a}{p} \right) = -1$?

\noindent An affirmative answer to this question will enable us to prove $\#A(K_1) = 2$ unconditionally (from Theorem \ref{Thm3}). Another consequence would be the capitulation of a nontrivial ideal class in the extension $K_1/\Q_1(\sqrt{p})$ (due to Theorem \ref{Thm2}). Although this condition is sufficient to ensure that $\#A(K_1) = \#A(\Q_1(\sqrt{p})) = 2$, there can still be an increase in the order of $A(K_2)$. To illustrate this, we enlist the orders of $A(K_2)$ and $A(\Q_2(\sqrt{p}))$ for small values of $r$ and $p$ satisfying Condition (\ref{Cond}) in the following tables. 

\begin{tabular}{c c}
Table 1: $r = 3$ & \hspace{4cm}  Table 2: $r=7$\\   
\aboverulesep=0ex 
\belowrulesep=0ex 
\begin{tabular}[t]{c|c|c}
\toprule
\rule{0pt}{1.1EM}
$p$ &  $\#A\left(\Q_2(\sqrt{p})\right)$ & $\#A(K_2)$ \\  
\midrule
\rule{0pt}{1.1EM}
41       &  4    & 4   \\
137      &  4    & 4   \\
521      &  2    & 4   \\
569      &  2    & 4   \\
761      &  4    & 4   \\
809      &  2    & 4   \\
857      &  2    & 4   \\
953      &  2    & 4   \\
\bottomrule
\end{tabular} &

\hspace{4cm}

\aboverulesep=0ex 
\belowrulesep=0ex 
\begin{tabular}[t]{c|c|c}
\toprule
\rule{0pt}{1.1EM}
$p$ &  $\#A\left(\Q_2(\sqrt{p})\right)$ & $\#A(K_2)$ \\  
\midrule
\rule{0pt}{1.1EM}
41       &  4    & 4   \\
313      &  4    & 4   \\
409      &  2    & 4   \\
521      &  2    & 4   \\
761      &  4    & 4   \\
857      &  2    & 4   \\
\bottomrule
\end{tabular}\\
\end{tabular}

%

We note that both the possibilities $\#A(K_2) = \#A(\Q_2(\sqrt{p}))$ and $\#A(K_2) = 2\cdot\#A(\Q_2(\sqrt{p}))$ can occur. Even for $r= 11, 19$, and $23$, it can be seen that $\#A(K_2) = 4$ for every $p \leq 1000$ that satisfies Condition (\ref{Cond}). Therefore, it would also be interesting to see what factors govern the growth of $\#A(K_2)$. 

{\bf Acknowledgements.} The authors warmly thank Prof. L. C. Washington for his guidance on the computational part of the article. The authors are grateful to Indian Institute of Technology Guwahati for providing excellent facilities to carry out this research.

\end{document}